\documentclass[12pt]{amsart}
\usepackage{a4wide,enumerate,xcolor}
\usepackage{amsmath,graphicx}
\allowdisplaybreaks

\let\pa\partial
\let\na\nabla
\let\eps\varepsilon
\newcommand{\N}{{\mathbb N}}
\newcommand{\R}{{\mathbb R}}
\newcommand{\diver}{\operatorname{div}}

\newtheorem{theorem}{Theorem}
\newtheorem{lemma}[theorem]{Lemma}

\newtheorem{remark}[theorem]{Remark}

\begin{document}

\title[Analysis of a Poisson--Nernst--Planck system]{Analysis of a Poisson--Nernst--Planck \\ cross-diffusion system with steric effects}

\author[P. Hirvonen]{Peter Hirvonen}
\address{Institute of Analysis and Scientific Computing, TU Wien, Wiedner Hauptstra\ss e 8--10, 1040 Wien, Austria}
\email{peter.hirvonen@tuwien.ac.at} 

\author[A. J\"ungel]{Ansgar J\"ungel}
\address{Institute of Analysis and Scientific Computing, TU Wien, Wiedner Hauptstra\ss e 8--10, 1040 Wien, Austria}
\email{juengel@tuwien.ac.at} 

\date{\today}

\thanks{The authors acknowledge partial support from   
the Austrian Science Fund (FWF), grants P33010 and F65.
This work has received funding from the European 
Research Council (ERC) under the European Union's Horizon 2020 research and innovation programme, ERC Advanced Grant NEUROMORPH, no.~101018153.} 

\begin{abstract}
A transient Poisson--Nernst--Planck system with steric effects is analyzed in a bounded domain with no-flux boundary conditions for the ion concentrations and mixed Dirichlet--Neumann boundary conditions for the electric potential. The steric repulsion of ions is modeled by a localized Lennard--Jones force, leading to cross-diffusion terms. The existence of global weak solutions, a weak--strong uniqueness property, and, in case of pure Neumann conditions, the exponential decay towards the thermal equilibrium state is proved. The main difficulties are the cross-diffusion terms and the different boundary conditions satisfied by the unknowns. These issues are overcome by exploiting the entropy structure of the equations and carefully taking into account the electric potential term. A numerical experiment illustrates the long-time behavior of the solutions when the potential satisfies mixed boundary conditions.
\end{abstract}

% \paragraph{Keywords:}  
 
% \paragraph{AMS classification:}  

\maketitle

%%%%%%%%%%%%%%%%%%%%%%%%%%%%%%%%%%%%%%%%%%%%%%%%%%%%%%%%%%%%%%%%%%%

\section{Introduction}

The transport of ions is often modeled by the Nernst--Planck equations for the ion concentrations, self-consistently coupled with the Poisson equation for the electric potential \cite{Ner89,Pla90}. The Nernst--Planck theory is valid for dilute solutions only. In confined environments, the ions become crowded and steric repulsion may appear due to finite ion sizes. In such a situation, the Nernst--Planck equations need to be modified. The mutual repulsive force can be described by the nonlocal Lennard--Jones force term \cite{HLLE12}. The localization limit leads to easier local interatomic forces. Then the free energy of the system consists of the local Lennard--Jones energy, the thermodynamic (Boltzmann) entropy, and the electric energy \cite{LiEi14}. In this paper, we analyze the associated Euler--Lagrange equations in a bounded domain with biologically motivated boundary conditions. We show the existence of global weak solutions, prove a weak--strong uniqueness property, show the exponential decay of the solutions towards thermal equilibrium, and present a numerical example.

The Nernst--Planck equations with steric effects for the ion concentrations $u_i$ are given by 
\begin{align}\label{1.eq}
  \pa_t u_i + \diver J_i = 0, \quad
  J_i = -\big(\sigma\na u_i + z_iu_i\na\Phi + u_i\na p_i(u)\big)
  \quad\mbox{in }\Omega,\ t>0,
\end{align}
where $\Omega\subset\R^d$ ($d\ge 1$) is a bounded domain with Lipschitz boundary and
\begin{align}\label{1.p}
  p_i(u) = \sum_{j=1}^n a_{ij}u_j, \quad i=1,\ldots,n,
\end{align}
describes the potential associated to the $i$th species arising from steric effects. The parameters are the diffusion coefficient $\sigma>0$, the ionic charges $z_i\in\R$, and the values $a_{ij}\ge 0$ represent the strength of steric repulsion between the species. To ensure the parabolicity of the system, we suppose that the matrix $(a_{ij})_{i,j=1}^n$ is symmetric and positive definite. We impose the initial and boundary conditions
\begin{align}\label{1.bic}
  u_i(0)=u_i^0 \quad\mbox{in }\Omega, \quad 
  J_i\cdot\nu = 0 \quad\mbox{on }\pa\Omega,\ t>0,\ i=1,\ldots,n,
\end{align}
where $\nu$ is the exterior unit normal vector to $\pa\Omega$. The electric potential $\Phi$ solves the Poisson equation with mixed Dirichlet--Neumann boundary conditions:
\begin{align}\label{1.Phi}
  -\Delta\Phi = \sum_{i=1}^n z_iu_i\quad\mbox{in }\Omega, \quad
  \Phi=\Phi_D\quad\mbox{on }\Gamma_D, \quad
  \na\Phi\cdot\nu = 0 \quad\mbox{on }\Gamma_N,
\end{align}
where $\Gamma_D\cup\Gamma_N=\pa\Omega$ and $\Gamma_D\cap\Gamma_D=\emptyset$. The right-hand side of \eqref{1.Phi} is the total ion charge. The no-flux boundary conditions mean that the ions cannot leave the domain, while the electric potential is assumed to be fixed on $\Gamma_D$, with $\Phi_D$ being the applied potential, and $\Gamma_N$ models insulating boundary parts.

\subsection{State of the art}

Without drift terms, equations \eqref{1.eq}--\eqref{1.p} have been rigorously derived from moderately interacting particle systems in a mean-field-type limit \cite{CDJ19}. A heuristic derivation including drift terms with a fixed potential for two species was presented in \cite{GaSe14} together with a global existence analysis. The Poisson--Nernst--Planck model \eqref{1.eq}--\eqref{1.p} was analyzed in \cite{Hsi19} for two species and (biologically less realistic) homogeneous Dirichlet boundary conditions for the electric potential. Equations \eqref{1.eq} for an arbitrary number of species and without potential terms were investigated in \cite{JPZ22}, but the existence proof was only sketched. An existence proof based on a finite-volume approximation for the same situation was presented in \cite{JuVe24}. An analysis of the $n$-species model \eqref{1.eq}--\eqref{1.Phi} is missing in the literature. In this paper, we fill this gap.

There exist other Nernst--Planck models modeling steric effects. In \cite{SWZ18}, the special case $a_{ij}=\delta_{ij}a_i$ ($\delta_{ij}$ is the Kronecker delta) was considered. This choice leads to nonlinear diffusion but does not contain cross-diffusion terms. Including the thermodynamic energy for the solvent in the free energy, the current densities become \cite{KBA07}
\begin{align}\label{1.solv}
  J_i = -\big(\sigma\na u_i - u_i\na\log u_0 + z_iu_i\na\Phi\big),
  \quad i=1,\ldots,n,
\end{align}
where $u_0=1-\sum_{j=1}^n u_j$ is the solvent concentration. In this model, the quantities $u_i$ are mass fractions, and their sum equals one. The corresponding Poisson--Nernst--Planck model was analyzed in \cite{JuMa24}, where the Poisson equation was replaced by the fourth-order Fermi--Poisson model \cite{SNE01}. Another approach is to allow for mobilities that depend on the solvent concentration, leading to
\begin{align}\label{1.solv2}
  J_i = -\big(\sigma u_0\na u_i - u_i\na u_0 
  + z_iu_0u_i\na\Phi\big), \quad i=1,\ldots,n,
\end{align}
which avoids the logarithmic term but involves the degenerate diffusion $u_0\na u_i$; we refer to \cite{BDPS10,SLH09} for a formal derivation and to \cite{GeJu18} for an existence analysis. The systems with fluxes \eqref{1.solv} or \eqref{1.solv2} yield a maximal concentration of ions, $u_i\le 1$. It is argued in \cite{Gil15} that local concentrations can be large. According to \cite[Sec.~5]{Gil15}, this does not contradict the intuition that the percentage of space occupied by the ions cannot exceed one, since we are averaging the concentrations over some small  neighborhood which keeps the local packing fraction below one. 

Finally, we mention the approach of \cite{Gav20}, where Nernst--Planck models with steric effects are derived by taking into account higher-order approximations of the Lennard--Jones term, leading to higher-order derivatives, which need to be added to equations \eqref{1.eq}. 

In all the mentioned existence results, the matrix $(a_{ij})$ is required to be positive definite. If the real part of at least one eigenvalue becomes negative for some value of the concentrations, and equations \eqref{1.eq} fail to be parabolic in the sense of Petrovskii, which is a minimal condition for local existence of solutions. In this situation, the problem is ill-posed \cite{Gav18}, and $L^\infty(\Omega)$ blow-up in finite time is possible \cite[Theorem 7]{GHLP23}. If the matrix is positive semidefinite, system \eqref{1.eq} is generally of hyperbolic--parabolic type. The case $a_{ij}=a>0$ and $n=2$ (with fixed potential terms) was investigated in \cite{BGHP85,GaSe14}. Formulated as a Lagrangian free-boundary problem, the two-species model may have more than one solution \cite{GSV15}. Without potential terms and for positive initial total concentrations, but for an arbitrary number of species, the existence and uniqueness of smooth solutions to the multi-species problem was proved in \cite{DrJu20}, while global dissipative measure-valued solutions were shown to exist in \cite{HoJu23}. We discuss the special case of rank-one matrices $(a_{ij})$ in Remark \ref{rem.one}.

The uniqueness of weak solutions to cross-diffusion systems is generally a delicate problem. Often, only the weak--strong uniqueness property is proved. This means that any weak solution coincides with a
strong solution with the same initial data, as long as the strong solution exists. This property was proved for, e.g., the Shigesada--Kawasaki--Teramoto population model \cite{ChJu19}, for a thin-film solar cell model \cite{BBEP20,HoBu22}, for Maxwell--Stefan systems \cite{HJT22}, and for fractional cross-diffusion equations \cite{DeZa24}, using the relative entropy method. Here, we extend this technique to equations \eqref{1.eq}--\eqref{1.p}. 

\subsection{Entropy structure and main difficulties}

The analytical results are based on variants of the entropy method. The entropy structure follows from the free energy, consisting of the thermodynamic (Boltzmann) entropy, the electric energy, and the mixing (Rao) entropy: 
\begin{align}\label{1.H}
  & H_{BR}(u) = \int_\Omega h_{BR}(u)dx, \quad\mbox{where} \\
  & h_{BR}(u) = \sigma\sum_{i=1}^n u_i(\log u_i-1) 
  + \frac12|\na(\Phi-\Phi_D)|^2
  + \frac12\sum_{i,j=1}^n a_{ij}u_iu_j
  + \sum_{i=1}^n z_iu_i\Phi_D. \nonumber 
\end{align}
We call $H_{BR}(u)$ the (mathematical) Boltzmann--Rao entropy. The last term compensates the Dirichlet boundary condition for the potential. System \eqref{1.eq}--\eqref{1.p} can be written as the formal gradient-flow equations
\begin{align*}
  \pa_t u_i = \diver\bigg(u_i\na\frac{\delta H_{BR}}{\delta u_i}\bigg),
  \quad i=1,\ldots,n,
\end{align*}
where $\delta H_{BR}/\delta u_i=\sigma\log u_i + z_i\Phi + p_i(u)$ is the variational derivative of $H_{BR}$ with respect to $u_i$ \cite[Lemma 7]{GeJu18}. This structure reveals that (see Theorem \ref{thm.ex} for a proof)
\begin{align*}
  \frac{d}{dt}H_{BR}(u) + \frac12\sum_{i=1}^n\int_\Omega 
  u_i|\na w_i|^2 dx \le 0,
\end{align*}
where $w_i=\delta H_{BR}/\delta u_i$ is called the entropy variable (electro-chemical potential in thermodynamics). To derive uniform gradient bounds, we need to estimate the entropy production term $\int_\Omega u_i|\na w_i|^2 dx$. We show in Lemma \ref{lem.ep} that 
\begin{align*}
  \sum_{i=1}^n\int_\Omega u_i|\na w_i|^2 dx
  \ge \sigma\sum_{i=1}^n\int_\Omega\big(4\sigma|\na\sqrt{u_i}|^2 
  + \alpha|\na u_i|^2\big)dx - C\sigma(H_{BR}(u)+1)
\end{align*}
for some constants $\alpha$, $C>0$. This is the key estimate for the existence analysis. Note that we lose the gradient bounds if $\sigma=0$. The entropy inequality yields an $L^\infty(0,T;L^2(\Omega))\cap L^2(0,T;H^1(\Omega))$ bound for $u_i$, an $L^\infty(0,T;H^1(\Omega))$ bound for $\Phi$, and an $L^{2}(0,T;L^{2}(\Omega))$ bound for $\sqrt{u_i}\na w_i$, which is sufficient for the existence analysis.

In the further analysis, we face two main difficulties. First, in the weak--strong uniqueness proof, to avoid issues with the logarithm, we do not use the Boltzmann--Rao entropy \eqref{1.H} but the Rao-type entropy
\begin{align*}%\label{1.HR}
  H_R(u) = \int_\Omega h_{R}(u)dx, \quad
  h_R(u) = \frac12|\na(\Phi-\Phi_D)|^2
  + \frac12\sum_{i,j=1}^n a_{ij}u_iu_j 
  + \sum_{i=1}^n z_iu_i\Phi_D.
\end{align*}
The time derivative of $H_R$ is formally computed by taking the test function $\phi_i=p_i(u)+z_i(\Phi-\Phi_D)$. Unfortunately, this function is not admissible since we need the regularity $\phi_i\in L^2(0,T;W^{1,4}(\Omega))$ in the weak formulation. This difficulty is overcome by extending the test function space and exploiting the regularity $\sqrt{u_i}\na w_i\in L^2(\Omega_T)$; see Section \ref{sec.ei} for details. Then the proof is based on the relative entropy 
\begin{align}\label{1.relH}
  H_{R}(u|\bar{u}) = \int_\Omega\big(h_R(u) - h_R(\bar{u}) 
  - h_R'(\bar{u})\cdot(u-\bar{u})\big)dx,
\end{align}
where $(u,\Phi)$ is a weak solution and $(\bar{u},\bar\Phi)$ is a strong solution. The goal is to derive an inequality of the form
\begin{align}\label{1.dHRdt}
  \frac{d}{dt}H_{R}(u|\bar{u}) \le CH_{R}(u|\bar{u}) 
  \quad\mbox{for }t>0,
\end{align}
where $C>0$ is some constant. The difficulty is to make this inequality rigorous. For this, we write the relative entropy as the sum of $H_R(u)$, $H_R(\bar{u})$, and some remainder terms. The solutions $u$ and $\bar{u}$ satisfy a Rao entropy inequality, while the remainder terms include the strong solution, which facilitates the computation of these expressions. Combining the estimations, we arrive eventually at \eqref{1.dHRdt}. 

The long-time behavior can be proved only if the stationary solution $(u^\infty,\Phi^\infty)$ is in thermal equilibrium, which means that the stationary entropy variable $w_i^\infty=\sigma\log u_i^\infty+z_i\Phi^\infty+p_i(u^\infty)$ vanishes. The relative entropy reads as
\begin{align}
  H_{BR}(u|u^\infty) &= \int_\Omega\big(h_{BR}(u) - h_{BR}(u^\infty)
  - h'_{BR}(u^\infty)\cdot(u-u^\infty)\big)dx. \label{1.Hinfty}
\end{align}
To prove exponential decay, we need an estimate like \eqref{1.dHRdt} but with $C<0$ to achieve an exponential decay rate via Gronwall's lemma. The fact that the boundary conditions of $u_i$ and $\Phi$ are different leads to the second difficulty. Indeed, computing the time derivative of the relative entropy, we find that
\begin{align*}
  \frac{d}{dt}H_{BR}(u|u^\infty) \le -\sum_{i=1}^n\int_\Omega
  u_i|\na(w_i-w_i^\infty)|^2 dx,
\end{align*}
and when expanding the square of $\na w_i$, the most delicate term is
\begin{align*}
  I := 2\sigma\sum_{i=1}^n\int_\Omega 
  z_i\na (u_i-u_i^\infty)\cdot\na(\Phi-\Phi^\infty)dx.
\end{align*}
The positive definiteness of $(a_{ij})$ provides the term $\|\na(u_i - u_i^\infty)\|_{L^2(\Omega)}^2$ with a good sign. Thus, we may apply Young's inequality for $I$ to absorb $\|\na(u_i - u_i^\infty)\|_{L^2(\Omega)}^2$, but the expression $\|\na(\Phi-\Phi^\infty)\|_{L^2(\Omega)}^2$ contributes to the right-hand side with the wrong sign. Another idea is to integrate by parts and to use the Poisson equation as well as the boundary condition $\na\Phi\cdot\nu=\na\Phi^\infty\cdot\nu=0$ on $\Gamma_N$:
\begin{align*}
  I = -2\sigma\int_\Omega \bigg(\sum_{i=1}^n z_i(u_i-u_i^\infty)
  \bigg)^2 dx
  + 2\sigma\sum_{i=1}^n\int_{\Gamma_D} (u_i-u_i^\infty)
  \na(\Phi-\Phi^\infty)\cdot\nu dx.
\end{align*}
The boundary integral would vanish if $u_i$ and $u_i^\infty$ were satisfying the same Dirichlet boundary conditions. This situation holds true in semiconductor applications, where the long-time behavior of solution could be shown \cite{Gaj85,GlHu97}. Unfortunately, this is not the case here. To overcome this issue, we assume that $\Phi$ satisfies Neumann conditions on the whole boundary, i.e.\ $\Gamma_D=\emptyset$. Then the delicate integral $I$ is nonpositive and we can conclude the argument. 

\subsection{Main results}

We impose the following assumptions:

\begin{itemize}
\item[(A1)] Domain: $\Omega\subset\R^d$ ($1\le d\le 4$) is a bounded domain with Lipschitz boundary $\pa\Omega=\Gamma_D\cup\Gamma_N$, where $\Gamma_D\cap\Gamma_N=\emptyset$ and $\Gamma_N$ is open in $\pa\Omega$.
\item[(A2)] Data: $T>0$, $\sigma>0$, $z_i\in\R$ for $i=1,\ldots,n$, and $(a_{ij})_{i,j=1}^n\subset\R^{n\times n}$ is symmetric and positive definite.
\item[(A3)] Initial and boundary data: $u_i^0\in L^2(\Omega)$ satisfies $u_i^0>0$ in $\Omega$. If $\pa\Omega=\Gamma_N$, we require that $\sum_{i=1}^n \int_\Omega z_iu_i^0 dx=0$.
\item[(A4)] Boundary data I: The boundary function $\Phi_D$ on $\Gamma_D$ can be extended to a function in $H^1(\Omega)\cap L^{\infty}(\Omega)$ satisfying $\Delta\Phi_D=0$ in $\Omega$ and $\na\Phi_D\cdot\nu=0$ on $\Gamma_N$.
\item[(A5)] Boundary data II: The solution $\phi$ to 
\begin{align*}
  -\Delta\phi=f\in L^2(\Omega)\quad\mbox{in }\Omega, 
  \quad \phi=\Phi_D\quad\mbox{on }\Gamma_D, \quad
  \na\phi\cdot\nu=0\quad\mbox{on }\Gamma_N
\end{align*}
satisfies $\phi\in H^2(\Omega)$ and $\|\phi\|_{H^2(\Omega)}\le C(\|f\|_{L^2(\Omega)}+1)$, where $C>0$ depends on $\Phi_D$.
\end{itemize}

Assumption (A1) includes the case of homogeneous Neumann boundary conditions, $\pa\Omega=\Gamma_N$. The restriction to $d\le 4$ space dimensions is due to Sobolev embeddings. It can be removed in the existence analysis at the expense of lower regularity of $u_i$. The positive definiteness of $(a_{ij})$ in Assumption (A2) helps us to derive $H^1(\Omega)$ estimates for the solution. We discuss a special case of positive semidefinite matrices (namely rank-one matrices) in Remark \ref{rem.one} and the case $\sigma=0$ in Remark \ref{rem.sigma}. Assumption (A4) is needed to define the Boltzmann--Rao entropy. The condition that $\Phi_D$ satisfies an elliptic problem is needed to compute the entropy variables \cite[Lemma 7]{GeJu18}. Finally, Assumption (A5) is only needed for the weak--strong uniqueness and exponential decay results. Assumption (A5) is a regularity requirement for the solutions to the mixed Dirichlet--Neumann boundary-value problem; they are satisfied if, for instance, the boundaries $\Gamma_D$ and $\Gamma_N$ do not meet and the boundary data is sufficiently smooth \cite{Tro87}.

We use the following notation. Set $\Omega_T:=\Omega\times(0,T)$. The norm $\|u\|_X$ of a vector-valued function $u=(u_1,\ldots,u_n)$ in some Banach space $X$ is defined as $\|u\|_X=\sum_{i=1}^n\|u_i\|_X$. Furthermore, we introduce the Hilbert space
\begin{align*}
  H_D^1(\Omega) = \{v\in H^1(\Omega):v=0\mbox{ on }\Gamma_D\}.
\end{align*}
Our first main result concerns the existence of global weak solutions. 

\begin{theorem}[Global existence]\label{thm.ex}
Let Assumptions (A1)--(A4) hold. Then there exists a weak solution $u=(u_1,\ldots,u_n)$ to \eqref{1.eq}--\eqref{1.Phi} satisfying $u_i\ge 0$ in $\Omega_T$,
\begin{align*}
  \sqrt{u_i},\,u_i,\,\Phi\in L^2(0,T;H^1(\Omega)), \quad 
  \pa_t u_i\in L^{2}(0,T;W^{1,4}(\Omega)'), \quad i=1,\ldots,n,
\end{align*}
and the initial data is satisfied in the sense of $W^{1,4}(\Omega)'$. If $\pa\Omega=\Gamma_N$, the potential satisfies $\int_\Omega \Phi dx = 0$. Moreover, the following entropy inequality holds:
\begin{align}\label{1.ei}
  H_{BR}(u(t)) + \int_0^t\int_\Omega u_i|\na w_i|^2 dxds 
  \le H_{BR}(u^0),
\end{align}
where $w_i = \sigma\log u_i + z_i\Phi + p_i(u)$ on $\{u_i>0\}$ and $w_i=0$ on $\{u_i=0\}$.
\end{theorem}

The theorem is proved by the Leray--Schauder fixed-point argument. For this, we approximate equations \eqref{1.eq} by using an implicit Euler discretization, adding a higher-order regularization, and formulating the equations in terms of the entropy variable $w_i$, as in the boundedness-by-entropy method \cite{Jue15}. The mapping $u\mapsto w$ is invertible and yields positive approximations of the concentrations; see Lemma \ref{lem.u}. Uniform estimates are obtained from the discrete version of the Boltzmann--Rao entropy inequality \eqref{1.ei}; see Lemmas \ref{lem.ei}--\ref{lem.unif}. The Aubin--Lions lemma of \cite{DrJu12} implies the strong convergence of a subsequence of approximating solutions, allowing us to identify the nonlinearities.

\begin{theorem}[Weak--strong uniqueness]\label{thm.wsu}
Let Assumptions (A1)--(A5) hold, let $(u,\Phi)$ be a weak solution and $(\bar{u},\bar\Phi)$ be a strong solution to \eqref{1.eq}--\eqref{1.Phi} in the sense $\na \bar{u}_i$, $\na\bar\Phi\in L^\infty(\Omega_T)$, and both solutions satisfy the entropy inequality \eqref{1.ei}. Then $u(t)=\bar{u}(t)$ and $\Phi(t)=\bar\Phi(t)$ in $\Omega$ for $t>0$.
\end{theorem}

As explained before, we use the relative entropy \eqref{1.relH} to prove the theorem. Given a weak solution $(u,\Phi)$ and a strong solution $(\bar{u},\bar\Phi)$, we write (see Lemma \ref{lem.relHR})
\begin{align*}
  \frac{d}{dt}H_R(u|\bar{u}) &= \frac{dH_R}{dt}(u) 
  + \frac{dH_R}{dt}(\bar{u}) - \frac12\frac{d}{dt}\sum_{i,j=1}^n
  \int_\Omega a_{ij}(u_i\bar{u}_j+\bar{u}_iu_j)dx \\
  &\phantom{xx}- \frac{d}{dt}\int_\Omega
  \na(\Phi-\Phi_D)\cdot\na(\bar\Phi-\Phi_D)dx.
\end{align*}
Compared to formulation \eqref{1.relH}, it is sufficient to suppose that $\bar{u}$ satisfies the entropy inequality, while the minus sign in \eqref{1.relH} in front of $H_R(\bar{u})$ requires that $\bar{u}$ fulfills the entropy equality. The time derivative of the remaining two integrals in the previous expression can be computed, as $(\bar{u},\bar\Phi)$ is assumed to have sufficient regularity. After some computations, we end up with
\begin{align*}
  \frac{d}{dt}H_R(u|\bar{u}) \le CH_R(u|\bar{u}),
\end{align*}
where $C>0$ depends on the $L^\infty(0,T;W^{1,\infty}(\Omega))$ norms of $\bar{u}$ and $\bar\Phi$. Gronwall's lemma then shows that $u(t)=\bar{u}(t)$ and $\Phi(t)=\bar\Phi(t)$ in $\Omega$ for $t>0$.

Finally, we prove the exponential convergence of the weak solutions to the thermal equilibrium state in case $\pa\Omega=\Gamma_N$. We call $(u^\infty,\Phi^\infty)$ a thermal equilibrium solution if the fluxes vanish, i.e.\ $\na(\sigma\log u_i^\infty + z_i\Phi^\infty + p_i(u^\infty))=0$ in $\Omega$, $i=1,\ldots,n$. A solution is given by $u_i^\infty=\operatorname{meas}(\Omega)^{-1}\int_\Omega u_i^0 dx$, i.e., $u^\infty$ is constant in space. Then
\begin{align*}
  0 = \na\big(\sigma\log u_i^\infty + z_i\Phi^\infty + p_i(u^\infty)\big)
  = z_i\na\Phi^\infty,
\end{align*}
and $\Phi^\infty$ is constant too. Because of the Neumann condition for $\Phi$, we have $\int_\Omega\Phi^\infty dx=0$. This implies that $\Phi^\infty=0$. In the numerical example of Section \ref{sec.num}, we present the convergence to a nonconstant steady state if mixed boundary condition for $\Phi$ are imposed.

\begin{theorem}[Exponential decay]\label{thm.long}
Let Assumptions (A1)--(A5) hold, let $\pa\Omega=\Gamma_N$, and let $(u,\Phi)$ be a weak solution to \eqref{1.eq}--\eqref{1.Phi}. Then there exists a constant $\lambda>0$
such that
\begin{align*}
  \|u(t)-u^\infty\|_{L^2(\Omega)} + \|\na\Phi(t)\|_{L^2(\Omega)}
  \le H_{BR}(u^0|u^\infty) e^{-\lambda\sigma t}, \quad t>0.
\end{align*}
\end{theorem}

The theorem is proved by differentiating the relative entropy \eqref{1.Hinfty} with respect to time. The entropy inequality \eqref{1.ei} shows that
\begin{align*}
  \frac{d}{dt}H_{BR}(u|u^\infty) 
  = \frac{dH_{BR}}{dt}(u)
  \le -\sum_{i=1}^n\int_\Omega u_i|\na w_i|^2 dx,
\end{align*}
and the goal is to estimate the entropy production from above in terms of the relative entropy. For this step, we need pure Neumann conditions for $\Phi$. Then, using the logarithmic Sobolev and Poincar\'e--Wirtinger inequalities, we end up with
\begin{align*}
  \frac{d}{dt}H_{BR}(u|u^\infty) \le -\lambda\sigma H_{BR}(u|u^\infty),
\end{align*}
where $\lambda>0$ depends on $\alpha$, $\max_{i,j=1,\ldots,n}a_{ij}$, and the constants of the logarithmic Sobolev and Poincar\'e--Wirtinger inequalities. Gronwall's lemma finishes the proof.

The paper is organized as follows. The existence result is proved in Section \ref{sec.ex}. Section \ref{sec.wsu} is concerned with the proof of the weak--strong uniqueness property, and the exponential decay is shown in Section \ref{sec.exp}. Finally, we present a numerical example in Section \ref{sec.num}.

%%%%%%%%%%%%%%%%%%%%%%%%%%%%%%%%%%%%%%%%%%%%%%%%%%%%%%%%%%%%%%%%%

\section{Global existence of solutions}\label{sec.ex}

The aim of this section is to prove Theorem \ref{thm.ex}. We first introduce an approximate problem in terms of the entropy variables, which is solved by means of the Leray--Schauder fixed-point theorem. Uniform estimates are derived from an approximate entropy inequality. These estimates allow us to use the Aubin--Lions compactness lemma to pass to the de-regularization limit. 

\subsection{Preparations}

First, we show that the mapping between entropy variables and densities is invertible. We introduce the ionic charge vector $z=(z_1,\ldots,z_n)\in\R^n$ .

\begin{lemma}\label{lem.u}
Let $\sigma>0$. Introduce the function $F:(0,\infty)^n\to\R^n$, $F_i(u)= \sigma\log u_i+p_i(u)$, where $p_i$ is defined in \eqref{1.p}. There exists a mapping $u:\R^n\times\R\to(0,\infty)^n$, $(w,\Phi)\mapsto u(w,\Phi)$, such that $u(w,\Phi)=F^{-1}(w-z\Phi)$. Moreover, $\Phi\mapsto \sum_{i=1}^n z_iu_i(w,\Phi)$ is a decreasing function.
\end{lemma}

\begin{proof} 
Taking into account that the range of the logarithm is the whole line, the range of $F$ equals $\R^n$. The function $F$ is strictly monotone since
\begin{align*}
  (u-v)\cdot(F(u)-F(v))
  = \sigma\sum_{i=1}^n (u_i-v_i)\log\frac{u_i}{v_i} 
  + \sum_{i,j=1}^n a_{ij}(u_i-v_i)(u_j-v_j) > 0
\end{align*}
for $u,v\in(0,\infty)^n$ with $u\neq v$ (this even holds if $(a_{ij})$ is positive semidefinite), showing that $F$ is one-to-one. Thus, the inverse $F^{-1}:\R^n\to(0,\infty)^n$ exists and is strictly monotone too. By the inverse function rule, it is differentiable and $(F^{-1})'(w)=(F'(u))^{-1}$. Moreover, $F'(u)$ and consequently $F'(u)^{-1}$ are positive definite. We define $u(w,\Phi):= F^{-1}(w-z\Phi)$. Then
\begin{align*}
  \frac{\pa}{\pa\Phi}\sum_{i=1}^n z_iu_i(w,\Phi)
  = -\sum_{i,j=1}^n z_iz_j (F'(u)^{-1})_{ij}
  \quad\mbox{with }u=u(w,\Phi)
\end{align*}
is negative, which means that $\Phi\mapsto \sum_{i=1}^n z_iu_i(w,\Phi)$ is decreasing.
\end{proof}

Lemma \ref{lem.u} shows that, with $u=u(w,\Phi)$,
\begin{align*}
  w_i = F_i(u) + z_i\Phi = \sigma\log u_i + p_i(u) + z_i\Phi, 
  \quad i=1,\ldots,n.
\end{align*}

\begin{lemma}
Let $w\in L^\infty(\Omega;\R^n)$ and $\Phi_D\in L^\infty(\Omega)$. Then 
there exists a unique weak solution $\Phi\in H^1(\Omega)\cap L^\infty(\Omega)$ to
\begin{align}\label{2.Phi}
  -\Delta\Phi = \sum_{i=1}^n z_iu_i(w,\Phi)\quad\mbox{in }\Omega,
  \quad \Phi=\Phi_D\quad\mbox{on }\Gamma_D, \quad
  \na\Phi\cdot\nu=0\quad\mbox{on }\Gamma_N.
\end{align}
\end{lemma}

\begin{proof}
The existence and uniqueness of $\Phi-\Phi_D\in H_D^1(\Omega)$ to \eqref{2.Phi} follows from the monotonicity of $\Phi\mapsto -\sum_{i=1}^n z_iu_i(w,\Phi)$. The boundedness of $\Phi$ is a consequence of the Stampacchia truncation method \cite[Sec.~2.3]{Tro87}. Indeed, let $m>m_0:=\|\Phi_D\|_{L^\infty(\Omega)}$ and use $(\Phi-m)^+=\max\{0,\Phi-m\}$ as a test function in the weak formulation of \eqref{2.Phi}:
\begin{align*}
  \int_\Omega&|\na(\Phi-m)^+|^2 dx
  = \sum_{i=1}^n\int_\Omega z_i\big(u_i(w,\Phi)-u_i(w,m)\big)
  (\Phi-m)^+ dx \\
  &\phantom{xx}+ \sum_{i=1}^n\int_\Omega z_i u_i(w,m)(\Phi-m)^+ dx
  \le \sum_{i=1}^n\int_\Omega z_i u_i(w,m_0)(\Phi-m)^+ dx \\
  &= C(\|w\|_{L^\infty(\Omega)})\|(\Phi-m)^+\|_{L^2(\Omega)}g(m)^{1/2}
  \le C\|\na(\Phi-m)^+\|_{L^2(\Omega)}
  g(m)^{1/2},
\end{align*}
where $g(m)=\operatorname{meas}(\{\Phi>m\})$ and we used the Poincar\'e inequality. This shows that
\begin{align*}
  \|\na(\Phi-m)^+\|_{L^2(\Omega)} \le C g(m)^{1/2}.
\end{align*}
By the Poincar\'e--Sobolev inequality, for $2<r<2d/(d-2)$ ($2<r<\infty$ if $d\le 2$) and $M>m$,
\begin{align*}
  C\|\na(\Phi-m)^+\|_{L^2(\Omega)}
  &\ge \|(\Phi-m)^+\|_{L^r(\Omega)}
  \ge \bigg(\int_{\{\Phi>M\}}[(\Phi-m)^+]^r dx\bigg)^{1/r} \\
  &\ge \bigg(\int_{\{\Phi>M\}}(M-m)^r dx\bigg)^{1/r} 
  = (M-m)g(M)^{1/r},
\end{align*}
and therefore, $g(M)\le C(M-m)^{-r} g(m)^{r/2}$. (In this paper, $C>0$ is a generic constant whose value may change from line to line.) Since $r/2>1$, it follows from Stampacchia's lemma \cite[Lemma 2.9]{Tro87} that there exists $M_0>0$ (depending on $m_0$) such that $g(M_0+m_0)=0$. Consequently, $\Phi\le M_1:=M_0+m_0$ in $\Omega$. We can prove $\Phi\ge M_2$ for some $M_2\in\R$ in a similar way.
\end{proof}

This result also holds for homogeneous Neumann boundary conditions if $\sum_{i=1}^n\int_\Omega z_iu_i^0 dx$ $=0$ in $\Omega$. (Instead of the Poincar\'e inequality, we need the Poincar\'e--Wirtinger inequality.) The solution $\Phi$ is unique if we require that $\int_\Omega\Phi dx = 0$.

%%%%%%%%%%%%%%%%%%

\subsection{Definition of an approximate problem}

Let $T>0$, $N\in\N$, $\tau=T/N>0$, and let $m\in\N$ be such that $m>d/2$. This implies that the embedding $H^m(\Omega)\hookrightarrow L^\infty(\Omega)$ is compact. We approximate the function $w(x,k\tau)$ by the piecewise constant in time function $w^k(x)$ and add a regularizing term. For this, let $w^{k-1}\in L^\infty(\Omega;\R^n)$ be given and let $\Phi^{k-1}-\Phi_D\in H^1_D(\Omega)\cap L^\infty(\Omega)$ be the unique solution to \eqref{2.Phi} with $w=w^{k-1}$. We introduce $u^{k-1}=u(w^{k-1},\Phi^{k-1})$, where $u(w,\Phi)$ is defined in Lemma \ref{lem.u}. Recall that $u_i^k>0$ in $\Omega$ by construction. In particular,
\begin{align}\label{2.w}
  w_i^k = \sigma\log u_i^k + z_i\Phi^k
  + \sum_{j=1}^n a_{ij}u_j^k, \quad
  \mbox{where }u_i^k = u_i(w^k,\Phi^k), \ i=1,\ldots,n.
\end{align}
If $k=1$, let $\Phi^0\in H^1(\Omega)$ be the unique solution to $-\Delta\Phi^0=\sum_{i=1}^n z_iu_i^0$ in $\Omega$ with the boundary conditions in \eqref{1.Phi} and set $w_i^0=\sigma\log u_i^0  + z_i\Phi^0 + p_i(u^0)$. 

We wish to find $w^k\in H^m(\Omega;\R^n)$ and $\Phi^k-\Phi_D\in H^1_D(\Omega)\cap L^\infty(\Omega)$ such that
\begin{align}\label{approx1}
  & \frac{1}{\tau}\int_\Omega
  \big(u(w^k,\Phi^k)-u(w^{k-1},\Phi^{k-1})\big)
  \cdot\phi dx \\
  &\phantom{xx}+ \sum_{i=1}^n\int_\Omega 
  u_i(w^k,\Phi^k)\na w_i^k\cdot\na\phi_i dx
  + \eps b(w^k,\phi) = 0, \nonumber \\
  & \int_\Omega\na\Phi^k\cdot\na\psi dx 
  = \int_\Omega\sum_{i=1}^n z_i u_i(w^k,\Phi^k)\psi dx \label{approx2}
\end{align}
for all $\phi\in H^m(\Omega;\R^n)$ and $\psi\in H_D^1(\Omega)$, where
\begin{align}\label{2.b}
  b(w^k,\phi) = \int_\Omega\bigg(\sum_{|\alpha|=m}D^\alpha w^k
  \cdot D^\alpha\phi + w^k\cdot\phi\bigg)dx,
\end{align}
$\alpha=(\alpha_1,\ldots,\alpha_d)\in\N_0^d$ is a multi-index, $|\alpha|=\alpha_1+\cdots+\alpha_d$, and $D^\alpha=\pa^{|\alpha|}/\pa x_1^{\alpha_1}\cdots\pa x_d^{\alpha_d}$ is a partial derivative. 

%%%%%%%%%%%%%%%%%%%%%%

\subsection{Approximate entropy inequality}

Assuming the existence of a weak solution $(w^k,$ $\Phi^k)$ to \eqref{approx1}--\eqref{approx2}, we derive the approximate entropy inequality which is used for the fixed-point argument. We set $u^k:=u(w^k,\Phi^k)$. 

\begin{lemma}[Approximate entropy inequality]\label{lem.ei}
Let $(w^k,\Phi^k)$ be a weak solution to \eqref{approx1}--\eqref{approx2}. Then, with $H_{BR}$ defined in \eqref{1.H},
\begin{align}\label{2.eiHBR}
  H_{BR}(u^k) + \tau\sum_{i=1}^n\int_\Omega u_i^k|\na w_i^k|^2 dx
  + C\eps\tau\|w^k\|_{H^m(\Omega)}^2
  \le H_{BR}(u^{k-1}).
\end{align}
\end{lemma}

\begin{proof}
We choose $w^k\in H^m(\Omega;\R^n)$ as a test function in \eqref{approx1} and use the generalized Poincar\'e inequality \cite[Chap.~2, Sec.~1.4]{Tem97}, which gives $b(w^k,w^k)\ge C\|w^k\|_{H^m(\Omega)}^2$, to find that
\begin{align*}%\label{2.aux}
  \int_\Omega(u^k-u^{k-1})\cdot w^k dx 
  + \tau\sum_{i=1}^n\int_\Omega u_i^k|\na w_i^k|^2 dx
  + C\eps\tau\|w^k\|_{H^m(\Omega)}^2 \le 0.
\end{align*}
We rewrite the first term as
\begin{align}\label{2.aux1}
  \int_\Omega(u^k&-u^{k-1})\cdot w^k dx
  = \sigma\sum_{i=1}^n\int_\Omega (u_i^k-u_i^{k-1})\log u_i^k dx \\
  &+ \sum_{i,j=1}^n\int_\Omega a_{ij}(u_i^k-u_i^{k-1})u_j^k dx
  + \sum_{i=1}^n\int_\Omega z_i(u_i^k-u_i^{k-1})(\Phi^k-\Phi_D) dx 
  \nonumber \\
  &+ \sum_{i=1}^n\int_\Omega z_i(u_i^k-u_i^{k-1})\Phi_D dx. \nonumber
\end{align}
We estimate the integrals on the right-hand side term by term. The convexity of $u\mapsto \sum_{i=1}^nu_i(\log u_i-1)$ implies that 
\begin{align*}
  \sigma\sum_{i=1}^n\int_\Omega (u_i^k-u_i^{k-1})\log u_i^k dx
  \ge \sigma\sum_{i=1}^n\int_\Omega\big(u_i^k(\log u_i^k-1)
  - u_i^{k-1}(\log u_i^{k-1}-1)\big) dx.
\end{align*}
We deduce from the symmetry and positive semidefiniteness of $(a_{ij})$ that
\begin{align*}
  \sum_{i,j=1}^n&\int_\Omega a_{ij}(u_i^k-u_i^{k-1})u_j^k dx
  = \sum_{i,j=1}^n\int_\Omega a_{ij}\bigg(u_i^ku_j^k 
  - \frac12 u_i^{k-1}u_j^k - \frac12 u_i^ku_j^{k-1}\bigg)dx \\
  &= \frac12\sum_{i,j=1}^n\int_\Omega a_{ij}(u_i^ku_j^k
  - u_i^{k-1}u_j^{k-1})dx + \frac12\sum_{i,j=1}^n\int_\Omega a_{ij}
  (u_i^k-u_i^{k-1})(u_j^k-u_j^{k-1})dx \\
  &\ge \frac12\sum_{i,j=1}^n\int_\Omega a_{ij}(u_i^ku_j^k
  - u_i^{k-1}u_j^{k-1})dx.
\end{align*}
Finally, taking into account the Poisson equation, we have
\begin{align*}
  \sum_{i=1}^n&\int_\Omega z_i(u_i^k-u_i^{k-1})(\Phi^k-\Phi_D)dx
  = \int_\Omega\na(\Phi^k-\Phi^{k-1})\cdot\na(\Phi^k-\Phi_D)dx \\
  &= \frac12\int_\Omega\big(|\na(\Phi^k-\Phi_D)|^2 
  - |\na(\Phi^{k-1}-\Phi_D)|^2\big)dx
  + \frac12\int_\Omega|\na(\Phi^k-\Phi^{k-1})|^2 dx \\
  &\ge \frac12\int_\Omega|\na(\Phi^k-\Phi_D)|^2 dx
  -\frac12\int_\Omega|\na(\Phi^{k-1}-\Phi_D)|^2dx.
\end{align*}
Hence, by the definition of $h_{BR}$ in \eqref{1.H}, it follows from \eqref{2.aux1} that
\begin{align*}
  \sum_{i=1}^n(u^k-u^{k-1})\cdot w^k dx 
  \ge \int_\Omega\big(h_{BR}(u^k)-h_{BR}(u^{k-1})\big)dx.
\end{align*}
We conclude that
\begin{align*}
  H_{BR}(u^k) - H_{BR}(u^{k-1}) + \tau\sum_{i=1}^n\int_\Omega
  u_i^k|\na w_i^k|^2 dx + C\eps\tau\|w^k\|_{H^m(\Omega)}^2 \le 0.
\end{align*}
This finishes the proof.
\end{proof}

We can derive uniform gradient bounds from the entropy inequality. This is the key step of the proof.

\begin{lemma}[Entropy production estimate]\label{lem.ep}
There exists $C>0$ independent of $(\eps,\tau)$ such that
\begin{align*}
  \sum_{i=1}^n\int_\Omega u_i^k|\na w_i^k|^2 dx
  &\ge \sum_{i=1}^n\int_\Omega\big(4\sigma^2|\na(u_i^k)^{1/2}|^2
  + \alpha\sigma|\na u_i^k|^2 \\
  &\phantom{xx}+ u_i^k|\na(p_i(u^k)+z_i\Phi^k)|^2\big)dx
  - C\sigma (H_{BR}(u^k)+1).
\end{align*}
\end{lemma}

\begin{proof}
We expand $u_i^k|\na w_i^k|^2$ by inserting definition \eqref{2.w} of $w^k_i$:
\begin{align*}
  \sum_{i=1}^n& u_i^k|\na w_i^k|^2 
  = \sigma^2\sum_{i=1}^n u_i^k|\na\log u_i^k|^2
  + \sum_{i=1}^n u_i^k|\na(p_i(u^k)+z_i\Phi^k)|^2 \\
  &\phantom{xx}+ 2\sigma\sum_{i=1}^n u_i^k
  \na\log u_i^k\cdot\na(p_i(u^k)+z_i\Phi^k) \\
  &= 4\sigma^2\sum_{i=1}^n|\na(u_i^k)^{1/2}|^2
  + \sum_{i=1}^n u_i^k|\na(p_i(u^k)+z_i\Phi^k)|^2 \\
  &\phantom{xx}+ 2\sigma\sum_{i,j=1}^n a_{ij}\na u_i^k\cdot\na u_j^k 
  + 2\sigma\sum_{i=1}^n z_i \na u_i^k\cdot\na\Phi^k \\
  &\ge 4\sigma^2\sum_{i=1}^n|\na(u_i^k)^{1/2}|^2
  + \sum_{i=1}^n u_i^k|\na(p_i(u^k)+z_i\Phi^k)|^2 \\
  &\phantom{xx}+ \alpha\sigma\sum_{i=1}^n|\na u_i^k|^2 
  - \frac{\sigma}{\alpha}\sum_{i=1}^n z_i^2|\na\Phi^k|^2,
\end{align*}
where we used $\sum_{i,j=1}^n a_{ij}\xi_i\xi_j\ge \alpha|\xi|^2$ for $\xi\in\R^n$ ($\alpha>0$ is the smallest eigenvalue of $(a_{ij})$) and Young's inequality in the last step. The last term is estimated according to
\begin{align*}
  \frac{\sigma}{\alpha}\sum_{i=1}^n z_i^2|\na\Phi^k|^2
  \le C\sigma (h_{BR}(u^k)+1),
\end{align*}
where $C>0$ depends on the $H^1(\Omega)$ norm of $\Phi_D$, which ends the proof.
\end{proof}

Lemmas \ref{lem.ei} and \ref{lem.ep} imply the following uniform bounds.

\begin{lemma}[Uniform bounds]\label{lem.unif}
Let $\tau>0$ be sufficiently small. Then there exists $C>0$ independent of $\eps$ and $\tau$ such that
\begin{align}\label{2.u}
  \sup_{k=1,\ldots,N}\|u^k\|_{L^2(\Omega)}
  &+ \tau\sum_{k=1}^N\bigg(\sum_{i=1}^n\|(u_i^k)^{1/2}\|_{H^1(\Omega)}^2 
  + \|u^k\|_{H^1(\Omega)}^2\bigg) \\
  &+ \tau\sum_{k=1}^N\bigg(\sum_{i=1}^n\|(u_i^k)^{1/2}
  \na w_i^k\|_{L^2(\Omega)}^2 + \eps\|w^k\|_{H^m(\Omega)}^2\bigg) \le C.
  \nonumber 
\end{align}
If Assumption (A5) holds then, for some constant $C>0$ independent of $\eps$ and $\tau$,
\begin{align}\label{2.up}
  \tau\sum_{i=1}^n\int_\Omega u_i^k\big(|\na p_i(u^k)|^2 
  + |\na\Phi^k|^2\big) dx \le C.
\end{align}
\end{lemma}

\begin{proof}
We insert the entropy production estimate of Lemma \ref{lem.ep} into the entropy inequality of Lemma \ref{lem.ei},
\begin{align*}
  (1&-\tau\sigma C)H_{BR}(u^k) 
  + \tau\sum_{i=1}^n\int_\Omega
  \big(4\sigma^2|\na(u_i^k)^{1/2}|^2
  + \alpha\sigma|\na u_i^k|^2 \\
  &+ u_i^k|\na(p_i(u^k)+z_i\Phi^k)|^2\big)dx
  + C\eps\tau\|w^k\|_{H^m(\Omega)}^2 
  \le H_{BR}(u^{k-1}) + C\tau\sigma,
\end{align*}
and sum the resulting inequality over $k=1,\ldots,j$ for $1<j\le N$:
\begin{align*}
  (1&-\tau\sigma C)H_{BR}(u^j) + \tau\sum_{k=1}^j\sum_{i=1}^n\int_\Omega
  \big(4\sigma^2|\na(u_i^k)^{1/2}|^2 + \alpha\sigma|\na u_i^k|^2 \\
  &\phantom{xx}+ u_i^k|\na(p_i(u^k)+z_i\Phi^k)|^2\big)dx
  + C\eps\tau\sum_{k=1}^j\|w^k\|_{H^m(\Omega)}^2 \\
  &\le H_{BR}(u^0) + C\sigma\tau + C\tau\sum_{k=1}^{j-1} H_{BR}(u^k).
\end{align*}
Choosing $0<\tau< 1/(C\sigma)$, the discrete Gronwall inequality \cite{Cla87} shows that for all $j\le N$,
\begin{align*}
  H_{BR}(u^j) \le \frac{H_{BR}(u^0) + C(T)}{1-\tau\sigma C}e^{CT},
\end{align*}
which leads to
\begin{align}
  &H_{BR}(u^j) + C\tau\sigma\sum_{k=1}^j\bigg(\sigma\sum_{i=1}^n
  \|\na(u_i^k)^{1/2}\|_{L^2(\Omega)}^2 
  + \|\na u^k\|_{L^2(\Omega)}^2\bigg)dx \label{2.HH} \\
  &\phantom{x}+ \tau\sum_{k=1}^j\sum_{i=1}^n\int_\Omega 
  u_i^k|\na(p_i(u^k)+z_i\Phi^k)|^2 dx 
  + C\eps\tau\sum_{k=1}^j\|w^k\|_{H^m(\Omega)}^2 \le C(u^0,T).
  \nonumber
\end{align}
We deduce from the positive definiteness of $(a_{ij})$ that $\|u^j\|_{L^2(\Omega)}$ is bounded uniformly in $(\eps,\tau)$ and uniformly for $j=1,\ldots,N$. Then estimate \eqref{2.u} follows from the Poincar\'e--Wirtinger inequality. 

Next, we prove estimate \eqref{2.up}. We infer from \eqref{2.HH} that
\begin{align}\label{2.aux3}
  \tau\sum_{k=1}^N&\sum_{i=1}^n\int_\Omega u_i^k|\na p_i(u^k)|^2 dx \\
  &\le 2\tau\sum_{k=1}^N\sum_{i=1}^n\int_\Omega 
  u_i^k|\na(p_i(u^k)+z_i\Phi^k)|^2 dx 
  + 2\tau\sum_{k=1}^N\sum_{i=1}^n\int_\Omega 
  z_i^2 u_i^k|\na\Phi^k|^2 dx \nonumber \\
  &\le C + C\tau\sum_{k=1}^N\sum_{i=1}^n\int_\Omega 
  u_i^k|\na\Phi^k|^2 dx. \nonumber 
\end{align}
We use the fact that the embedding $H^2(\Omega)\hookrightarrow W^{1,4}(\Omega)$ is continuous for $d\le 4$ to estimate
\begin{align*}
  \tau\sum_{k=1}^N\int_\Omega u_i^k|\na\Phi^k|^2 dx
  &\le \tau\sum_{k=1}^N\|u_i^k\|_{L^2(\Omega)}
  \|\na\Phi^k\|_{L^4(\Omega)}^2 \\
  &\le C\tau\sum_{k=1}^N\|u_i^k\|_{L^2(\Omega)}
  \|\Phi^k\|_{H^2(\Omega)}^2.
\end{align*}
To bound the $H^2(\Omega)$ norm of $\Phi^k$, we deduce from Assumption (A5) that $\|\Phi^k\|_{H^2(\Omega)} \le C\sum_{i=1}^n\|u_i^k\|_{L^2(\Omega)} + C(\Phi_D)$. Hence, 
\begin{align*}%\label{2.aux4}
  \tau\sum_{k=1}^N\sum_{i=1}^n\int_\Omega u_i^k|\na\Phi^k|^2 dx
  &\le C\tau\sum_{k=1}^N\sum_{i=1}^n\|u_i^k\|_{L^2(\Omega)}^3 + C(T) \\
  &\le CT\sum_{i=1}^n\Big(\sup_{k=1,\ldots,N}\|u_i^k\|_{L^2(\Omega)}
  \Big)^3 + C(T) \le C(T).
  \nonumber
\end{align*}
Inserting this estimate into \eqref{2.aux3} concludes the proof.
\end{proof}

%%%%%%%%%%%%%%%%%%

\subsection{Solution of the approximate problem}

We show that problem \eqref{approx1}--\eqref{approx2} possesses a weak solution.

\begin{lemma}
Let $w^{k-1}\in L^\infty(\Omega;\R^n)$ be given and let $\Phi^{k-1}-\Phi_D\in H^1_D(\Omega)\cap L^\infty(\Omega)$ be the unique solution to the Poisson equation \eqref{2.Phi} with $w=w^{k-1}$. Then, for sufficiently small $\tau>0$, there exists a solution $w^k\in H^m(\Omega;\R^n)$, $\Phi^k-\Phi_D\in H^1_D(\Omega)\cap L^\infty(\Omega)$ to \eqref{approx1}--\eqref{approx2}. 
\end{lemma}

\begin{proof}
The idea is to use the Leray--Schauder fixed-point theorem. Let $y\in L^\infty(\Omega;\R^n)$ and $\delta\in[0,1]$. Let $\Phi^k-\Phi_D\in H_D^1(\Omega)\cap L^\infty(\Omega)$ be the unique solution to 
\begin{align*}
  \int_\Omega\na\Phi^k\cdot\na\psi dx
  = \int_\Omega\sum_{i=1}^n z_iu_i(y,\Phi^k)\psi dx
\end{align*}
for test functions $\psi\in H_D^1(\Omega)$, where $u_i(y,\Phi)$ is defined in Lemma \ref{lem.u}. Next, consider the linear problem
\begin{align}\label{2.a}
  a(v,\phi) = F(\phi) \quad\mbox{for all }\phi\in H^m(\Omega;\R^n),
\end{align}
where
\begin{align*}
  a(v,\phi) &= \delta\sum_{i=1}^n\int_\Omega u_i(y,\Phi^k)
  \na v_i\cdot\na\phi_i dx + \eps b(v,\phi), \\ 
  F(\phi) &= -\frac{\delta}{\tau}\int_\Omega
  \big(u(y,\Phi^k)-u(w^{k-1},\Phi^{k-1})\big)\cdot\phi dx,
\end{align*}
recalling definition \eqref{2.b} of $b$. By the generalized Poincar\'e inequality, the bilinear form $a$ is coercive on $H^m(\Omega)$:
\begin{align*}
  a(v,v) = \delta\sum_{i=1}^n\int_\Omega u_i(y,\Phi^k)|\na v_i|^2 dx
  + \eps b(v,v) \ge C\eps\|v\|_{H^m(\Omega)}^2.
\end{align*}
Moreover, $a$ and $F$ are continuous since $u_i(y,\Phi^k)\in L^\infty(\Omega)$. By the Lax--Milgram lemma, there exists a unique $y\in H^m(\Omega;\R^n)\hookrightarrow L^\infty(\Omega;\R^n)$ to \eqref{2.a}.

This defines the fixed-point operator $S:L^\infty(\Omega;\R^n)\times[0,1]\to L^\infty(\Omega;\R^n)$. It holds that $S(y,0)=0$ for all $y\in L^\infty(\Omega;\R^n)$. The continuity of $S$ can be shown as in proof of Lemma 5 in \cite{Jue15}. The compactness of $S$ follows from the compactness of the embedding $H^{m}(\Omega;\R^n)\hookrightarrow L^\infty(\Omega;\R^n)$. It remains to determine a uniform estimate for all fixed points of $S(\cdot,\delta)$. Let $w^k\in H^m(\Omega;\R^n)$ be such a fixed point. If $\delta=0$, there is nothing to show. Hence, let $\delta>0$. Proceeding as in the proof of Lemma \ref{lem.ei}, we obtain the inequality
\begin{align*}
  \delta H_{BR}(u^k) + C\eps\tau\|w^k\|_{H^m(\Omega)}^2
  \le \delta H_{BR}(u^{k-1}) \le H_{BR}(u^{k-1}),
\end{align*}
where $u^k=u(w^k,\Phi^k)$. Choosing $0<\tau<1/(C\sigma)$, this provides a uniform estimate for $w^k$ in $H^m(\Omega;\R^n)$ and consequently in $L^\infty(\Omega;\R^n)$, which is the desired uniform bound. The Leray--Schauder theorem implies the existence of a fixed point of $S(\cdot,1)$, which is a solution to \eqref{approx1}, where $\Phi^k$ solves \eqref{approx2}.
\end{proof}

%%%%%%%%%%%%%%%%%%%%%

\subsection{Limit $\eps\to 0$}

We perform first the limit $\eps\to 0$. (The existence of a weak solution can be proved by performing the simultaneous limit $(\eps,\tau)\to 0$, but we need the limit $\eps\to 0$ to show an inequality for the entropy $H_R$.) For this, we fix $k\in\{1,\ldots,n\}$ and set $w_i^\eps:=w_i^k$, $\Phi^\eps:=\Phi^k$, and $u_i^\eps:=u_i^k$. The uniform estimates of Lemma \ref{lem.unif} imply the existence of a subsequence, which is not relabeled, such that
\begin{align*}
  & (u_i^\eps)^{1/2}\rightharpoonup \sqrt{u_i}, \
  u_i^\eps\rightharpoonup u_i \quad\mbox{weakly in }H^1(\Omega), 
  \quad \na\Phi^\eps\rightharpoonup \na\Phi
  \quad\mbox{weakly in }L^2(\Omega), \\
  & u_i^\eps\to u_i \quad\mbox{strongly in }L^2(\Omega), \quad
  \eps w_i^\eps \to 0 \quad\mbox{strongly in }H^m(\Omega)\mbox{ as }
  \eps\to 0
\end{align*}
for $i=1,\ldots,n$. This shows that
\begin{align*}
  (u_i^\eps)^{1/2}\na w_i^\eps
  &= 2\sigma\na(u_i^\eps)^{1/2} 
  + z_i(u_i^\eps)^{1/2}\na\Phi^\eps
  + (u_i^\eps)^{1/2}\sum_{j=1}^n a_{ij}\na u_j^\eps \\
  &\rightharpoonup 2\sigma\na\sqrt{u_i} 
  + z_i\sqrt{u_i}\na\Phi 
  + \sqrt{u_i}\sum_{j=1}^n a_{ij}\na u_j
  = \sqrt{u_i}\na w_i
\end{align*}
weakly in $L^{4/3}(\Omega)$, where we defined
\begin{align*}
  w_i = \sigma\log u_i + z_i\Phi
  + \sum_{j=1}^n a_{ij}u_j \quad\mbox{on }
  \{u_i>0\}, \quad w_i=0 \quad\mbox{on }\{u_i=0\}.
\end{align*}
In fact, in view of the uniform $L^2(\Omega)$ bound for $(u_i^\eps)^{1/2}\na w_i^\eps$ from Lemma \ref{lem.unif}, the weak convergence holds even in $L^2(\Omega)$. Notice that $\na w_i$ may not exist but $\sqrt{u_i}\na w_i$ is a function in $L^2(\Omega)$. Since $u_i\in H^1(\Omega)\hookrightarrow L^4(\Omega)$ for $d\le 4$, we have $u_i\na w_i\in L^{4/3}(\Omega)$. 

The convergences allow us to pass to the limit in the approximate problem \eqref{approx1}--\eqref{approx2}, showing that $u^k:=u\in L^2(\Omega;\R^n)$ and $\Phi^k-\Phi_D:=\Phi-\Phi_D\in H_D^1(\Omega)$ is a solution to 
\begin{align}\label{approx3}
  & \frac{1}{\tau}\int_\Omega(u^k-u^{k-1})\cdot\phi dx
  + \sum_{i=1}^n\int_\Omega u_i^k\na w_i^k\cdot\na\phi_i dx = 0, \\
  & \int_\Omega\na\Phi^k\cdot\na\psi dx = \int_\Omega\sum_{i=1}^n
  z_iu_i^k\psi dx \label{approx4}
\end{align} 
for all $\phi\in W^{1,4}(\Omega;\R^n)$ and $\psi\in H_D^1(\Omega)$. We wish to pass to the limit $\eps\to 0$ in the entropy inequality of Lemma \ref{lem.ei}. For this, we recall that $u_i^\eps\to u_i^k$ strongly in $L^2(\Omega)$, and $\na\Phi^\eps\to\na\Phi^k$, $(u_i^\eps)^{1/2}\na w_i^\eps\rightharpoonup (u_i^k)^{1/2}\na w_i^k$ weakly in $L^2(\Omega)$. Therefore, by the weakly lower semicontinuity of convex continuous functions, the limit $\eps\to 0$ in \eqref{2.eiHBR} leads to
\begin{align}\label{2.eiHBR2}
  H_{BR}(u^k) + \tau\sum_{i=1}^n\int_\Omega u_i^k|\na w_i^k|^2 dx 
  \le H_{BR}(u^{k-1}).
\end{align}

\begin{lemma}\label{lem.HR}
It holds that
\begin{align*}
  H_R(u^k) &+ \tau\sigma\sum_{i,j=1}^n\int_\Omega a_{ij}
  \na u_i^k\cdot\na u_j^k dx 
  + \tau\sum_{i=1}^n\int_\Omega
  u_i^k\big|\na(p_i(u^k)+z_i\Phi^k)\big|^2 dx \\
  &\le H_R(u^{k-1}) 
  - \tau\sigma\sum_{i=1}^n\int_\Omega z_i\na u_i^k
  \cdot\na\Phi^kdx,
\end{align*}
recalling definition \eqref{1.H} of $H_R$.
\end{lemma}

\begin{proof}
We wish to use $\phi_i=p_i(u^k)+z_i\Phi^k$ as a test function in \eqref{approx3} but this function is not an element of $W^{1,4}(\Omega)$. In fact, we can extend the test function space for \eqref{approx3}. Let $V_i^k:=\{\psi\in H^1(\Omega):(u_i^k)^{1/2}\na\psi\in L^2(\Omega)\}$. Then $W^{1,4}(\Omega)\subset V_i^k$ (here we use $u_i^k\in L^2(\Omega)$) and $W^{1,4}(\Omega)$ is dense in $V_i^k$. Therefore, we can replace the test function space $W^{1,4}(\Omega)$ by $V_i^k$. By estimate \eqref{2.up} (here we need Assumption (A5)), $p_i(u^k)$ and $\Phi^k$ are elements of $V_i^k$, such that we can use $\phi_i$ as a test function in \eqref{approx3}. Then
\begin{align*}
  & 0 = I_1+I_2+I_3, \quad\mbox{where} \\
  & I_1 = \sum_{i=1}^n\int_\Omega(u_i^k-u_i^{k-1})p_i(u^k)dx, \quad
  I_2 = \sum_{i=1}^n\int_\Omega z_i(u_i^k-u_i^{k-1})\Phi^k dx, \\
  & I_3 = \tau\sum_{i=1}^n\int_\Omega u_i^k\na w_i^k
  \cdot\na(p_i(u^k)+z_i\Phi^k)dx.
\end{align*}
As in the proof of \eqref{2.aux1}, we find that
\begin{align*}
  I_1 &= \sum_{i,j=1}^n\int_\Omega a_{ij}(u_i^k-u_i^{k-1})u_j^k dx
  \ge \frac12\sum_{i,j=1}^n\int_\Omega a_{ij}(u_i^ku_j^k
  - u_i^{k-1}u_j^{k-1})dx.
\end{align*}
We use the Poisson equation \eqref{approx4} and integrate by parts to estimate the term $I_2$:
\begin{align*}
  I_2 &= \sum_{i=1}^n\int_\Omega(u_i^k-u_i^{k-1})(\Phi^k-\Phi_D)dx
  + \sum_{i=1}^n\int_\Omega z_i(u_i^k-u_i^{k-1})\Phi_D dx \\
  &= -\int_\Omega\Delta(\Phi^k-\Phi^{k-1})(\Phi^k-\Phi_D)dx 
  + \sum_{i=1}^n\int_\Omega z_i(u_i^k-u_i^{k-1})\Phi_D dx \\
  &= \int_\Omega\na\big((\Phi^k-\Phi_D)-(\Phi^{k-1}-\Phi_D)\big)
  \cdot\na(\Phi^k-\Phi_D)dx 
  + \sum_{i=1}^n\int_\Omega z_i(u_i^k-u_i^{k-1})\Phi_D dx \\
  &\ge \bigg(\frac12\int_\Omega|\na(\Phi^k-\Phi_D)|^2dx
  + \sum_{i=1}^n\int_\Omega z_i u_i^k\Phi_D dx\bigg) \\
  &\phantom{xx}- \bigg(\frac12\int_\Omega|\na(\Phi^{k-1}-\Phi_D)|^2dx
  + \sum_{i=1}^n\int_\Omega z_i u_i^{k-1}\Phi_D dx\bigg).
\end{align*}
This shows that $I_1+I_2\ge H_R(u^k)-H_R(u^{k-1})$. Finally, we rewrite $I_3$, decomposing $w_i^k=\sigma\log u_i^k + (p_i(u^k)+z_i\Phi^k)$:
\begin{align*}
  I_3 &= \tau\sum_{i=1}^n\int_\Omega u_i^k\na\big(\sigma\log u_i^k
  + (p_i(u^k)+z_i\Phi^k)\big)\cdot\na\big((p_i(u^k)+z_i\Phi^k)\big)dx \\
  &= \tau\sum_{i=1}^n\int_\Omega
  u_i^k|\na(p_i(u^k)+z_i\Phi^k)|^2 dx
  + \tau\sigma\sum_{i,j=1}^n\int_\Omega 
  a_{ij}\na u_i^k\cdot\na u_j^k dx \\
  &\phantom{xx}+ \tau\sigma\sum_{i=1}^n\int_\Omega
  z_i\na u_i^k\cdot\na\Phi^k dx,
\end{align*}
ending the proof.
\end{proof}

%%%%%%%%%%%%%%%%%%%%%

\subsection{Limit $\tau\to 0$}

We introduce the piecewise constant in time functions $u_i^{(\tau)}(x,t)=u_i^k(x)$, $w_i^{(\tau)}(x,t)=w_i^k(x)$, and $\Phi^{(\tau)}(x,t)=\Phi^k(x)$ for $x\in\Omega$ and $t\in((k-1)\tau,k\tau]$ and the shift operator $(\sigma_\tau u^{(\tau)})(\cdot,t)=u^{k-1}$ for $t\in((k-1)\tau,k\tau]$. Then equations \eqref{approx1}--\eqref{approx2} become
\begin{align}\label{tau1}
  & \frac{1}{\tau}\int_0^T\int_\Omega
  \big(u^{(\tau)}-\sigma_\tau u^{(\tau)}\big)\cdot\phi dxdt
  + \sum_{i=1}^n\int_0^T\int_\Omega u_i^{(\tau)}\na w_i^{(\tau)}
  \cdot\na\phi_i dxdt = 0, \\
  & \int_0^T\int_\Omega\na\Phi^{(\tau)}\cdot\na\psi dxdt
  = \int_0^T\int_\Omega\sum_{i=1}^n z_iu_i^{(\tau)}\psi dxdt
  \label{tau2} 
\end{align}
for $\phi_i\in L^2(0,T;W^{1,4}(\Omega))$ and $\psi\in L^2(0,T;H^1_D(\Omega))$. We obtain from Lemma \ref{lem.unif} the uniform estimates
\begin{align}
  \|u_i^{(\tau)}\|_{L^\infty(0,T;L^2(\Omega))}
  + \|u_i^{(\tau)}\|_{L^2(0,T;H^1(\Omega))} 
  + \|(u_i^{(\tau)})^{1/2}\|_{L^2(0,T;H^1(\Omega))} &\le C, \nonumber \\
  \|(u_i^{(\tau)})^{1/2}\na w_i^{(\tau)}\|_{L^2(\Omega_T)}
  + \|\Phi^{(\tau)}\|_{L^\infty(0,T;H^1(\Omega))}
  + \sqrt{\eps}\|w_i^{(\tau)}\|_{L^2(0,T;H^m(\Omega))} &\le C, 
  \nonumber \\
  \|(u_i^{(\tau)})^{1/2}\na p_i(u^{(\tau)})\|_{L^2(\Omega_T)}
  + \|(u_i^{(\tau)})^{1/2}\na\Phi^{(\tau)}\|_{L^2(\Omega_T)}
  &\le C, \label{2.pp}
\end{align}
where $C>0$ depends on $T$ but not on $\eps$ and $\tau$. Estimate \eqref{2.pp} holds under Assumption (A5) and is not needed for the existence analysis but for the derivation of the Rao-type entropy inequality. 

We need a uniform bound for the discrete time derivative. 
%We infer from the Gagliardo--Nirenberg inequality with $\theta=d/(d+2)$ and $r=(2d+4)/d$ that
%\begin{align}\label{2.Lr}
%  \|u_i^{(\tau)}\|_{L^r(\Omega_T)}^r
%  &\le C\int_0^T\|u_i^{(\tau)}\|_{H^1(\Omega)}^{r\theta}
%  \|u_i^{(\tau)}\|_{L^2(\Omega)}^{r(1-\theta)}dt \\
%  &\le \|u_i^{(\tau)}\|_{L^\infty(0,T;L^2(\Omega))}^{r(1-\theta)}
%  \int_0^T\|u_i^{(\tau)}\|_{H^1(\Omega)}^2 dt \le C. \nonumber 
%\end{align}

\begin{lemma}[Discrete time estimate]\label{lem.time}
There exists a constant $C>0$ independent of $\eps$ and $\tau$ such that
\begin{align*}
  \tau^{-1}\|u_i^{(\tau)}-\sigma_\tau u_i^{(\tau)}\|_{L^{2}(0,T;W^{1,4}(\Omega)')} \le C.
\end{align*}
\end{lemma}

\begin{proof}
We use $\phi_i\in L^2(0,T;W^{1,4}(\Omega))$ as a test function in \eqref{tau1}:
\begin{align*}
  \frac{1}{\tau}\bigg|\int_0^T&\int_\Omega(u_i^{(\tau)}
  - \sigma_\tau u_i^{(\tau)})\phi_i dxdt\bigg| 
  = \bigg|\int_0^T\int_\Omega u_i^{(\tau)}\na w_i^{(\tau)}
  \cdot\na\phi_i dxdt\bigg| \\
  &\le \|(u_i^{(\tau)})^{1/2}\|_{L^\infty(0,T;L^4(\Omega))}
  \big\|(u_i^{(\tau)})^{1/2}\na w_i^{(\tau)}\big\|_{L^2(\Omega_T)}
  \|\na\phi_i\|_{L^2(0,T;L^4(\Omega))} \\
  &\le C\|\na\phi_i\|_{L^2(0,T;L^4(\Omega))},
\end{align*} 
which finishes the proof.
\end{proof}

We apply the Aubin--Lions lemma in the version of \cite{DrJu12} to conclude from the $L^2(0,T;$ $H^1(\Omega))$ bound for $u^{(\tau)}$ and Lemma \ref{lem.time} that there exists a subsequence, which is not relabeled, such that
\begin{align*}
  u_i^{(\tau)}\to u_i \quad\mbox{strongly in }L^2(\Omega_T)
  \ \mbox{as }\tau\to 0.
\end{align*}
Moreover, we obtain
\begin{align*}
  \tau^{-1}(u_i^{(\tau)}-\sigma_\tau u_i^{(\tau)})
  \rightharpoonup \pa_t u_i &\quad\mbox{weakly in }
  L^{2}(0,T;W^{1,4}(\Omega)'), \\
  u_i^{(\tau)}\rightharpoonup u_i, \quad
  \Phi^{(\tau)}\rightharpoonup\Phi &\quad\mbox{weakly in }
  L^2(0,T;H^1(\Omega)).
\end{align*}
The linearity of $p_i$ implies that
\begin{align*}
  p_i(u^{(\tau)}) \rightharpoonup p_i(u)\quad\mbox{weakly in }
  L^2(0,T;H^1(\Omega)). 
\end{align*}
We infer that
\begin{align*}
  (u_i^{(\tau)}&)^{1/2}\na w_i^{(\tau)}
  = 2\sigma\na(u_i^{(\tau)})^{1/2} 
  + z_i(u_i^{(\tau)})^{1/2}\na\Phi^{(\tau)}
  + (u_i^{(\tau)})^{1/2}\na p_i(u^{(\tau)}) \\
  &\rightharpoonup 2\sigma\na\sqrt{u_i} 
  + z_i\sqrt{u_i}\na\Phi
  + \sqrt{u_i}\na p_i(u) \quad\mbox{weakly in }
  L^{2}(0,T;L^{4/3}(\Omega)).
\end{align*}
In fact, by Lemma \ref{lem.unif}, this convergence holds in $L^2(\Omega_T)$. This implies that
\begin{align*}
  u_i^{(\tau)}\na w_i^{(\tau)} &= (u_i^{(\tau)})^{1/2}\cdot
  (u_i^{(\tau)})^{1/2}\na w_i^{(\tau)} \\
  &\rightharpoonup\sigma\na u_i + z_i u_i\na\Phi + u_i\na p_i(u)
  \quad\mbox{weakly in }L^2(0,T;L^{4/3}(\Omega)).  
\end{align*}

Passing to the limit $\tau\to 0$ in \eqref{tau1}--\eqref{tau2} yields
\begin{align*}
  & \int_0^T\langle\pa_t u_i,\phi_i\rangle dt
  + \int_0^T\int_\Omega u_i\na w_i\cdot\na\phi_i dxdt = 0, \\
  & \int_0^T\int_\Omega\na\Phi\cdot\na\psi dxdt
  = \int_0^T\int_\Omega \sum_{i=1}^n z_iu_i\psi dxdt
\end{align*}
for test functions $\phi_i\in L^{2}(0,T;W^{1,4}(\Omega))$ and $\psi\in L^2(0,T;H_D^1(\Omega))$, where $i=1,\ldots,n$ and $w_i$ is defined by
\begin{align*}
  w_i = \sigma\log u_i + z_i\Phi + p_i(u)
  \quad\mbox{on }\{u_i>0\}, \quad
  w_i = 0 \quad\mbox{on }\{u_i=0\}.
\end{align*}
It can be shown as in Step 3 of the proof of Lemma 5 in \cite{Jue15} that the initial conditions are satisfied in the sense of $W^{1,4}(\Omega)'$.

%%%%%%%%%%%%%%%%%%%

\subsection{Entropy inequalities}\label{sec.ei}

We sum the entropy inequality \eqref{2.eiHBR2} over $k=1,\ldots,j$ with $j\le N$:
\begin{align*}
  H_{BR}(u^j) + \tau\sum_{k=1}^j\sum_{i=1}^n u_i^k|\na w_i^k|^2 dx
  \le H_{BR}(u^0).
\end{align*}
In terms of the piecewise constant in time function $u^{(\tau)}$, this inequality reads as
\begin{align*}
  H_{BR}(u^{(\tau)}(t)) + \sum_{i=1}^n\int_0^t\int_\Omega 
  u_i^{(\tau)}|\na w_i^{(\tau)}|^2 dxds \le H_{BR}(u^0).
\end{align*}
The strong $L^2(\Omega_T)$ convergence of $(u^{(\tau)})$ as well as the weak $L^2(\Omega_T)$ convergences of $(\na\Phi^{(\tau)})$ and $((u_i^{(\tau)})^{1/2}\na w_i^{(\tau)})$ lead to
\begin{align*}
  H_{BR}(u(t)) + \int_0^t\int_\Omega u_i|\na w_i|^2 dxds 
  \le H_{BR}(u^0).
\end{align*}
This completes the proof of Theorem \ref{thm.ex}.

We wish to pass to the limit $\tau\to 0$ in the entropy inequality of Lemma \ref{lem.HR}, which is needed for the weak--strong uniqueness property. For this, we require Assumption (A5). Summing this inequality over $k=1,\ldots,j$ with $j\le N$ and writing the inequality in terms of $u^{(\tau)}$, we have
\begin{align*}
  H_R&(u^{(\tau)}(t)) - H_R(u^0)
  + \sigma\sum_{i,j=1}^n\int_0^t\int_\Omega a_{ij}
  \na u_i^{(\tau)}\cdot\na u_j^{(\tau)}dxds \\
  &\phantom{xx}+ \sum_{i=1}^n\int_0^t\int_\Omega 
  u_i^{(\tau)}\big|\na(p_i(u^{(\tau)}) + z_i\Phi^{(\tau)})
  \big|^2 dxds \\
  &\le -\sigma\sum_{i=1}^n\int_0^t\int_\Omega 
  z_i \na u_i^{(\tau)}\cdot\na\Phi^{(\tau)}dxds.
\end{align*}
Using similar arguments as before, the limit $\tau\to 0$ gives
\begin{align}\label{2.eiR}
  H_R(u(t)) &+ \sigma\sum_{i,j=1}^n\int_0^t\int_\Omega a_{ij}
  \na u_i\cdot\na u_j dxds + \sum_{i=1}^n\int_0^t\int_\Omega 
  u_i\big|\na(p_i(u)+z_i\Phi)\big|^2 dxds \\
  &\le H_R(u^0) - \sigma\sum_{i=1}^n\int_0^t\int_\Omega 
  z_i\na u_i\cdot\na\Phi dxds. \nonumber
\end{align}

\begin{remark}[Rank-one case]\label{rem.one}\rm
We have assumed that the matrix $(a_{ij})$ is positive definite. One may ask whether we can also treat the case of positive semidefinite matrices. An idea is to use the technique of \cite{BGHP85}, extended in \cite{DrJu20}. We only consider matrices $(a_{ij})$ of rank one, i.e.\ $a_{ij}=a>0$ for $i,j=1,\ldots,n$. We also assume that $\sigma=0$ and $z_i=z_0\in\R$ for $i=1,\ldots,n$. We introduce the change of unknowns $v:=\sum_{i=1}^n u_i$ and $v_i:=u_i/v$ for $i=1,\ldots,n$. Summing equations \eqref{1.eq} over $i=1,\ldots,n$ gives the nonlinear drift-diffusion equation
\begin{equation}\label{pme}
\begin{aligned}
  & \pa_t v = \diver(av\na v + z_0v\na\Phi)
  \quad\mbox{in }\Omega,\ t>0, \\
  & \na v\cdot\nu = 0\quad\mbox{on }\pa\Omega, \quad
  v(\cdot,0) = \sum_{i=1}^n u_i^0\quad\mbox{in }\Omega, 
\end{aligned}
\end{equation}
where the potential $\Phi$ solves $-\Delta\Phi=z_0v$ in $\Omega$ with mixed Dirichlet--Neumann boundary conditions. A computation shows that the relative concentrations satisfy the transport equation
\begin{align*}
  \pa_t v_i = U\cdot\na v_i, \quad\mbox{where } U = \na(av + z_0\Phi).
\end{align*}
The transport equation can be solved by the method of characteristics if the transport velocity $U$ is bounded and $\diver U$ is H\"older continuous. Thus, we need sufficient regularity for problem \eqref{pme}. Because of the porous-medium term in \eqref{pme}, we cannot expect classical solutions in general. However, the solutions are classical (for smooth initial data) if the potential is sufficiently regular and $v$ is strictly positive. There are two issues. First, the mixed boundary conditions generally prevent $\Phi$ to be a classical. This can be resolved by assuming that the Dirichlet and Neumann boundary parts do not meet \cite{Tro87}. Second, the positive lower bound is usually proved by using the Stampacchia truncation method. Unfortunately, this is delicate because of the different boundary conditions for $v$ and $\Phi$. We leave details to future work.
\qed\end{remark}

\begin{remark}[Vanishing linear diffusion]\label{rem.sigma}\rm
The existence result may be proved for $\sigma=0$. In this case, the natural entropy is given by $H_R$, but the associated entropy inequality does not give suitable gradient estimates. We use instead the functional \begin{align*}
  H_1(u) = \int_\Omega\bigg(\sum_{i=1}^n u_i(\log u_i-1)
  + \frac12|\na(\Phi-\Phi_D)|^2 
  + \frac12\sum_{i,j=1}^n a_{ij} u_iu_j
  + \sum_{i=1}^n z_iu_i\Phi_D\bigg)dx.
\end{align*}
In case $\sigma=0$, the entropy variable becomes $w_i=z_i\Phi+p_i(u)$. Then, formally differentiating and applying Young's inequality,
\begin{align*}
  \frac{dH_1}{dt}(u)
  &= -\sum_{i=1}^n\int_\Omega u_i\na(z_i\Phi+p_i(u))
  \cdot\na\big(\log u_i+z_i\Phi+p_i(u)\big)dx \\
  &= -\sum_{i=1}^n\int_\Omega z_i\na u_i\cdot\na\Phi dx
  - \sum_{i,j=1}^n\int_\Omega a_{ij}\na u_i\cdot\na u_j dx \\
  &\phantom{xx}- \sum_{i=1}^n\int_\Omega u_i
  \big|\na(p_i(u)+z_i\Phi)\big|^2 dx \\
  &\le -\frac{\alpha}{2}\sum_{i=1}^n\int_\Omega|\na u_i|^2 dx
  + C\int_\Omega|\na\Phi|^2dx.
\end{align*}
Similarly as in the proof of Theorem \ref{thm.ex}, we arrive at the estimate
\begin{align*}
  \frac{dH_{BR}}{dt}(u) + \frac{\alpha}{2}\sum_{i=1}^n\int_\Omega
  |\na u_i|^2 dx\le C(H_{BR}(u)+1),
\end{align*}
which yields $L^2(0,T;H^1(\Omega))$ bounds for $u_i$ and an $L^\infty(0,T;H^1(\Omega))$ bound for $\Phi$. 

There is still an issue with the inversion of the mapping $(w,\Phi)\mapsto u(w,\Phi)$, since the logarithm is needed to ensure the positivity of $u_i$. This problem can be overcome by adding the term $\delta\log u_i$, i.e.\ $w_i = \delta\log u_i+z_i\Phi + p_i(u)$. Compared to the case $\sigma>0$, we have to pass to the limit $\delta\to 0$; see, e.g., the proof of \cite[Theorem 4]{Jue15} for details. Still, we need to show that the limit $\delta\to 0$ is possible in the modified approximate equations, and we leave details to the reader. 
\qed\end{remark}

%%%%%%%%%%%%%%%%%%%%%%%%%%%%%%%%%%%%%%%%%%%%%%%%%%%%%%%%%%%%%%%%%

\section{Weak--strong uniqueness}\label{sec.wsu}

In this section, we show Theorem \ref{thm.wsu}. First, we rewrite the relative entropy \eqref{1.relH}.

\begin{lemma}\label{lem.relHR}
It holds that
\begin{align*}
  H_R(u|\bar{u}) &= H_R(u) + H_R(\bar{u})
  - \frac12\sum_{i=1}^n\int_\Omega a_{ij}
  (u_i\bar{u}_j + \bar{u}_i u_j)dx \\
  &\phantom{xx}- \int_\Omega\na(\Phi-\Phi_D)
  \cdot\na(\bar\Phi-\Phi_D)dx. \nonumber 
\end{align*}
\end{lemma}.

\begin{proof}
By definition \eqref{1.relH} of $H_R(u|\bar{u})$,
\begin{align*}
  H_R(u|\bar{u}) &= \int_\Omega\bigg(\frac12|\na(\Phi-\Phi_D)|^2
  - \frac12|\na(\bar\Phi-\Phi_D)|^2
  + \frac12\sum_{i,j=1}^n a_{ij}(u_iu_j - \bar{u}_i\bar{u}_j) \\
  &\phantom{xx}+ \sum_{i=1}^n z_i(u_i-\bar{u}_i)\Phi_D\bigg)dx
  - \sum_{i=1}^n\int_\Omega
  \bigg(z_i\bar\Phi + \sum_{j=1}^n a_{ij}\bar{u}_j\bigg)
  (u_i-\bar{u}_i) dx \\
  &= \int_\Omega\bigg(\frac12|\na(\Phi-\Phi_D)|^2
  - \frac12|\na(\bar\Phi-\Phi_D)|^2
  + \frac12\sum_{i,j=1}^n a_{ij}(u_iu_j + \bar{u}_i\bar{u}_j)\bigg)dx \\
  &\phantom{xx}- \sum_{i=1}^n\int_\Omega\bigg(z_i(u_i-\bar{u}_i)
  (\bar\Phi-\Phi_D) - \sum_{i,j=1}^n a_{ij}u_i\bar{u}_j 
  \bigg)dx.
\end{align*}
We use the Poisson equation to reformulate the term involving $\bar\Phi-\Phi_D$:
\begin{align*}
  -\sum_{i=1}^n&\int_\Omega z_i(u_i-\bar{u}_i)(\bar\Phi-\Phi_D)dx
  = \int_\Omega\Delta(\Phi-\bar\Phi)(\bar\Phi-\Phi_D)dx \\
  &= \int_\Omega\Delta(\Phi-\Phi_D)(\bar\Phi-\Phi_D)dx
  - \int_\Omega\Delta(\bar\Phi-\Phi_D)(\bar\Phi-\Phi_D)dx \\
  &= -\int_\Omega\na(\Phi-\Phi_D)\cdot\na(\bar\Phi-\Phi_D)dx
  + |\na(\bar\Phi-\Phi_D)|^2 dx.
\end{align*}
Thus, by the symmetry of $(a_{ij})$,
\begin{align*}
  H_R(u|\bar{u}) &= \int_\Omega\bigg(\frac12|\na(\Phi-\Phi_D)|^2
  + \frac12\sum_{i,j=1}^n a_{ij}u_iu_j\bigg)dx \\
  &\phantom{xx}+ \int_\Omega\bigg(\frac12|\na(\bar\Phi-\Phi_D)|^2
  + \frac12\sum_{i,j=1}^n a_{ij}\bar{u}_i\bar{u}_j\bigg) dx \\
  &\phantom{xx}-\int_\Omega\na(\Phi-\Phi_D)\cdot\na(\bar\Phi-\Phi_D)dx
  - \frac12\sum_{i,j=1}^n\int_\Omega a_{ij}
  (u_i\bar{u}_j+\bar{u}_i u_j)dx. 
\end{align*}
Inserting the definitions of $H_R(u)$ and $H_R(\bar{u})$ concludes the proof.
\end{proof}

The solutions $u$ and $\bar{u}$ satisfy the entropy inequality \eqref{2.eiR}. Furthermore, since the mixed terms $u_i\bar{u}_j+\bar{u}_iu_j$ contain the strong solution and $(a_{ij})$ is symmetric, we can compute
\begin{align*}
  -\frac12&\sum_{i,j=1}^n\int_\Omega a_{ij}
  (u_i\bar{u}_j + \bar{u}_i u_j)dx\Big|_{0}^{t}
  = -\sum_{i,j=1}^n\int_0^t a_{ij}\big(\langle \pa_t u_i,\bar{u}_j\rangle
  + \langle\pa_t\bar{u}_i,u_j\rangle\big)ds \\
  &= -\sum_{i=1}^n\int_0^t\big(\langle\pa_t u_i,p_i(\bar{u})\rangle
  + \langle\pa_t\bar{u}_i,p_i(u)\rangle\big)ds \\
  &= \sum_{i=1}^n\int_0^t\int_\Omega\big(u_i\na w_i\cdot\na p_i(\bar{u})
  + \bar{u}_i\na\bar{w}_i\cdot\na p_i(u)\big)dxds.
\end{align*}
where $\bar{w}_i=\log\bar{u}_i+z_i\bar\Phi+p_i(\bar{u})$ on $\{\bar{u}_i>0\}$ and $\bar{w}_i=0$ on $\{\bar{u}_i=0\}$. We insert the definitions of $w_i$ and $\bar{w}_i$:
\begin{align}\label{3.mix1}
  -\frac12\sum_{i,j=1}^n&\int_\Omega a_{ij}
  (u_i\bar{u}_j + \bar{u}_i u_j)dx\Big|_{0}^{t}
  = \sigma\sum_{i,j=1}^n\int_0^t\int_\Omega a_{ij}(
  \na u_i\cdot\na\bar{u}_j + \na\bar{u}_i\cdot\na u_j)dxds \\
  &+ \sum_{i=1}^n\int_0^t\int_\Omega 
  \big(z_iu_i\na\Phi\cdot\na p_i(\bar{u})
  + z_i\bar{u}_i\na\bar\Phi\cdot\na p_i(u) \nonumber \\
  &+ (u_i+\bar{u}_i)\na p_i(u)\cdot\na p_i(\bar{u})\big)dxds. \nonumber
\end{align}
Furthermore, using the Poisson equation,
\begin{align}\label{3.mix2}
  -&\int_\Omega\na(\Phi-\Phi_D)
  \cdot\na(\bar\Phi-\Phi_D)dx\Big|_{0}^{t} \\
  &= -\int_0^t\int_\Omega\big(\na\pa_t\Phi\cdot\na(\bar\Phi-\Phi_D)
  + \na(\Phi-\Phi_D)\cdot\na\pa_t\bar\Phi\big) dxds  \nonumber \\
  &= -\sum_{i=1}^n\int_0^t z_i\big(\langle\pa_t u_i,\bar\Phi
  -\Phi_D\rangle
  + \langle\pa_t\bar{u}_i,\Phi-\Phi_D\rangle\big)ds \nonumber \\
  &= \sum_{i=1}^n\int_0^t\int_\Omega 
  z_i u_i\na w_i\cdot\na(\bar\Phi-\Phi_D)dxds
  + \sum_{i=1}^n\int_0^t\int_\Omega z_i\bar{u}_i\na\bar{w}_i
  \cdot\na(\Phi-\Phi_D)dxds \nonumber \\
  &= \sum_{i=1}^n\int_0^t\int_\Omega z_i\big(\sigma\na u_i
  + u_i\na q_i\big)\cdot\na(\bar\Phi-\Phi_D) dxds \nonumber \\
  &\phantom{xx}
  + \sum_{i=1}^n\int_0^t\int_\Omega z_i\big(\sigma\na\bar{u}_i
  + \bar{u}_i\na \bar{q}_i\big)\cdot\na(\Phi-\Phi_D) dxds, \nonumber 
\end{align}
where we have set $q_i:=p_i(u)+z_i\Phi$ and $\bar{q}_i:=p_i(\bar{u})+z_i\bar\Phi$. We add the inequalities for $H_R(u(t))$ and $H_R(\bar{u}(t))$ as well as the identities \eqref{3.mix1} and \eqref{3.mix2}. Then some terms can be combined and after some computations, we end up with
\begin{align}
  H_R(u&(t)|\bar{u}(t)) - H_R(u^0|\bar{u}^0)
  + \sigma\sum_{i,j=1}^n\int_0^t\int_\Omega a_{ij}\na(u_i-\bar{u}_i)
  \cdot\na(u_j-\bar{u}_j)dxds \nonumber \\
  &\le -\sigma\sum_{i=1}^n\int_0^t\int_\Omega z_i\na(u_i-\bar{u}_i)\cdot
  \na(\Phi-\bar\Phi)dxds \label{3.aux}  \\
  &\phantom{xx}
  - \sum_{i=1}^n\int_0^t\int_\Omega\big\{u_i|\na q_i|^2
  + \bar{u}_i|\na\bar{q}_i|^2 
  - (u_i+\bar{u}_i)\na q_i\cdot\na\bar{q}_i
  \big\}dxds \nonumber \\
  &= -\sigma\sum_{i=1}^n\int_0^t\int_\Omega z_i\na(u_i-\bar{u}_i)\cdot
  \na(\Phi-\bar\Phi)dxds \nonumber \\
  &\phantom{xx}- \sum_{i=1}^n\int_0^t\int_\Omega
  u_i|\na(q_i-\bar{q}_i)|^2 dxds 
  - \sum_{i=1}^n\int_0^t\int_\Omega (u_i-\bar{u}_i)\na\bar{q}_i\cdot
  \na(q_i-\bar{q}_i)dxds.
  \nonumber 
\end{align}
By the positive definiteness, we have
\begin{align*}
  \sigma\sum_{i,j=1}^n\int_0^t\int_\Omega a_{ij}\na(u_i-\bar{u}_i)
  \cdot\na(u_j-\bar{u}_j)dxds 
  \ge \alpha\sigma\sum_{i=1}^n\int_0^t\int_\Omega
  |\na(u_i-\bar{u}_i)|^2 dxds.
\end{align*}
We use Young's inequality for the first term on the right-hand side of \eqref{3.aux}:
\begin{align*}
  - \sigma\sum_{i=1}^n\int_0^t\int_\Omega z_i\na(u_i-\bar{u}_i)\cdot
  \na(\Phi-\bar\Phi)dxds
  &\le \frac{\alpha\sigma}{2}\sum_{i=1}^n\int_0^t\int_\Omega
  |\na(u_i-\bar{u}_i)|^2 dxds \\
  &\phantom{xx}+ C(\alpha)\sigma\int_\Omega |\na(\Phi-\bar\Phi)|^2 dxds.
\end{align*}
The second term on the right-hand side of \eqref{3.aux} is nonpositive and can be neglected. Finally, the last term in \eqref{3.aux} is estimated according to 
\begin{align*}
  -&\sum_{i=1}^n\int_0^t\int_\Omega (u_i-\bar{u}_i)\na\bar{q}_i\cdot
  \na(q_i-\bar{q}_i)dxds \\
  &\le \sum_{i=1}^n\|u_i-\bar{u}_i\|_{L^2(\Omega_T)}
  \|\na\bar{q}_i\|_{L^\infty(\Omega_T)} 
  \big\|\na(p_i(u)-p_i(\bar{u}))
  + z_i\na(\Phi-\bar\Phi)\big\|_{L^2(\Omega)} \\
  &\le C\sum_{i=1}^n\|u_i-\bar{u}_i\|_{L^2(\Omega_T)}
  \|\na\bar{q}_i\|_{L^\infty(\Omega_T)} \big(\|\na(u_i-\bar{u}_i)\|_{L^2(\Omega_T)}
  + \|\na(\Phi-\bar\Phi)\|_{L^2(\Omega_T)}\big) \\
  &\le \frac{\alpha\sigma}{2}\sum_{i=1}^n
  \|\na(u_i-\bar{u}_i)\|_{L^2(\Omega)}^2
  + C\sum_{i=1}^n\|u_i-\bar{u}_i\|_{L^2(\Omega)}^2
  + C\|\na(\Phi-\bar\Phi)\|_{L^2(\Omega)}^2,
\end{align*}
where $C>0$ depends on the $L^\infty(\Omega_T)$ norm of $\na\bar{q}_i$. Adding these estimates and taking into account that
\begin{align*}
  \sum_{i=1}^n\|u_i-\bar{u}_i\|_{L^2(\Omega)}^2
  + \|\na(\Phi-\bar\Phi)\|_{L^2(\Omega)}^2
  \le CH_{R}(u|\bar{u}),
\end{align*}
we conclude from \eqref{3.aux} that 
\begin{align*}
  H_R(u(t)|\bar{u}(t)) = H_R(u(t)|\bar{u}(t)) - H_R(u^0|\bar{u}^0)
  \le C\int_0^t H_R(u|\bar{u})ds,
\end{align*}
where we used the fact that the initial data of the weak and strong solutions coincide by assumption, $u^0=\bar{u}^0$. It follows from Gronwall's inequality that $H(u(t)|\bar{u}(t))=0$ for $t>0$ and consequently $u(t)=\bar{u}(t)$ and $\Phi(t)=\bar\Phi(t)$ in $\Omega$ for $t>0$. This proves Theorem \ref{thm.wsu}. 

%%%%%%%%%%%%%%%%%%%%%%%%%%%%%%%%%%%%%%%%%%%%%%%%%%%%%%%%%%%%%%%%%

\section{Exponential decay}\label{sec.exp}

We prove Theorem \ref{thm.long}. Recall that $\na\Phi\cdot\nu=0$ on $\pa\Omega$ in this proof. Differentiating the relative entropy $H_{BR}(u|u^\infty)$, defined in \eqref{1.Hinfty}, with respect to time and taking into account mass conservation and the fact that $u^\infty$ is constant, we find that 
\begin{align*}
  \frac{d}{dt}H_{BR}(u|u^\infty) = \frac{d}{dt}H_{BR}(u).
\end{align*}
We infer from the entropy inequality \eqref{1.ei} that
\begin{align}
  \frac{d}{dt}H_{BR}(u|u^\infty)
  &\le -\sum_{i=1}^n\int_\Omega u_i|\na w_i|^2 dx \label{4.dHdt} \\
  &= -\sum_{i=1}^n\int_\Omega
  \big(4\sigma^2|\na\sqrt{u_i}|^2 + u_i|\na(p_i(u)+z_i\Phi)|^2\big)dx 
  \nonumber \\
  &\phantom{xx}- 2\sigma\sum_{i=1}^n\int_\Omega
  \big(\na u_i\cdot\na p_i(u) + z_i\na u_i\cdot\na\Phi\big)dx.
  \nonumber 
\end{align}
(Strictly speaking, this inequality does not follow directly from \eqref{1.ei}. In fact, we can derive a similar inequality by replacing $H_{BR}(u^0)$ by $H_{BR}(u(s))$ with $0<s<t$. Then dividing by $s-t$ and passing to the limit $s\to t$ yields \eqref{4.dHdt}.) Taking into account the positive definiteness, we obtain 
\begin{align*}
  -2\sigma\sum_{i=1}^n\int_\Omega \na u_i\cdot\na p_i(u) dx
  = -2\sigma\sum_{i,j=1}^n\int_\Omega a_{ij}\na u_i\cdot\na u_j dx
  \le -2\alpha\sigma\sum_{i=1}^n\int_\Omega|\na u_i|^2 dx.
\end{align*}
It follows from the Poincar\'e--Wirtinger inequality, using $u_i^\infty=\operatorname{meas}(\Omega)^{-1}\int_\Omega u_i dx$ (by mass conservation), that 
\begin{align*}
  -2\sigma\sum_{i=1}^n\int_\Omega \na u_i\cdot\na p_i(u) dx
  &\le -2\alpha\sigma C_P^{-1}\sum_{i=1}^n
  \int_\Omega(u_i-u_i^\infty)^2 dx \\
  &\le -C(a)\sigma\sum_{i,j=1}^n\int_\Omega a_{ij}(u_i-u_i^\infty)
  (u_j-u_j^\infty),
\end{align*}
where $C(a)=2\alpha/(C_P\max_{i,j=1,\ldots,n}|a_{ij}|)$. We use the Poisson equation and the boundary condition $\na\Phi\cdot\nu=0$ on $\pa\Omega$ to rewrite the last term on the right-hand side of \eqref{4.dHdt}:
\begin{align*}
  -2\sigma\sum_{i=1}^n\int_\Omega z_i\na u_i\cdot\na\Phi dx
  &= 2\sigma\int_\Omega\sum_{i=1}^n z_iu_i\Delta\Phi dx
  = -2\sigma\int_\Omega\bigg(\sum_{i=1}^n z_iu_i\bigg)^2 dx \\
  &\le -2\sigma\int_\Omega|\na\Phi|^2 dx,
\end{align*}
where the last step is the usual elliptic estimate. By the logarithmic Sobolev inequality \cite[Rem.~2.6]{Jue16},
\begin{align*}
  -4\sigma^2\sum_{i=1}^n\int_\Omega|\na\sqrt{u_i}|^2 dx
  \le -4C_L\sigma^2\sum_{i=1}^n\int_\Omega u_i\log\frac{u_i}{u_i^\infty}dx.
\end{align*}
We combine the previous estimates to conclude from \eqref{4.dHdt} that
\begin{align*}
  \frac{d}{dt}H_{BR}(u|u^\infty)
  &\le -\sum_{i=1}^n\int_\Omega\bigg(
  4C_L\sigma^2 u_i\log\frac{u_i}{u_i^\infty}
  + 2\sigma|\na \Phi|^2 \\
  &\phantom{xx}+ C(a)\sigma\sum_{i,j=1}^n\int_\Omega 
  a_{ij}(u_i-u_i^\infty)(u_j-u_j^\infty)\bigg)dx \\
  &\le -\sigma\min\{2,4C_L\sigma,C(a)\}H_{BR}(u|u^\infty).
\end{align*}
Applying Gronwall's lemma proves the theorem. Notice that the decay rate vanishes if $\sigma=0$. 

%%%%%%%%%%%%%%%%%%%%%%%%%%%%%%%%%%%%%%%%%%%%%%%%%%%%%%%%%%%%%%%%%

\section{Numerical experiment}\label{sec.num}

We present a numerical result obtained by the software Netgen/NGSolve \cite{Sch14}, where we use mixed Dirichlet--Neumann boundary conditions for the Poisson equation. The domain is the square $\Omega=(0,1)^2$ with Dirichlet conditions on the left and right sides of the square and Neumann conditions on the top and bottom sides. We consider three species. The potential on the left and right sides of the square is defined by $\Phi_D=0.1$. The parameters are chosen as $\sigma=1$, $z_1=z_3=-5$, $z_2=5$, and the matrix
\begin{align*}
  (a_{ij}) = \begin{pmatrix}
  2.5 & 1 & 1 \\ 1 & 1 & 0.5 \\ 1 & 0.5 & 0.5
  \end{pmatrix}.
\end{align*}
is positive definite. We choose the initial data
\begin{align*}
  u_i(x,y) = \exp\big(-100(x-x_i)^2-100(y-y_i)^2\big) + 0.5, \quad
  i=1,2,3,
\end{align*}
where $(x_1,y_1)=(0.25,0.75)$, $(x_2,y_2)=(0.5,0.5)$, and $(x_3,y_3)=(0.75,0.25)$. For the numerical test, we have taken the mesh size 0.05 and the time step size $4\times 10^{-5}$. Figure \ref{fig1} shows the concentrations at the time steps $N=0$, $N=30$, and $N=380$. The last value corresponds to a solution that is close to the equilibrium state. At time step $N=380$, the solution is rather flat in the interior of the domain, while the largest variations can be observed at the (Dirichlet) boundary. This is consistent with the numerical experiments of \cite[Fig.~1A]{Gav18}, where surface charge effects were observed at the boundary. We observe that the shape of the profile depends on the sign of the valence. The electric potential does not change significantly over time, as can be seen in Figure \ref{fig2}. This example shows that the concentrations converge to a (non-constant) steady state as $t\to\infty$ in the case of mixed boundary conditions (which is expected). We note that the software is very sensible in situations where the drift or cross-diffusion terms are dominant such that in these cases a structure-preserving numerical scheme (to ensure the positivity of the densities) needs to be developed. This is left for future work.

\begin{figure}[ht]
\centering
\includegraphics[width=150mm]{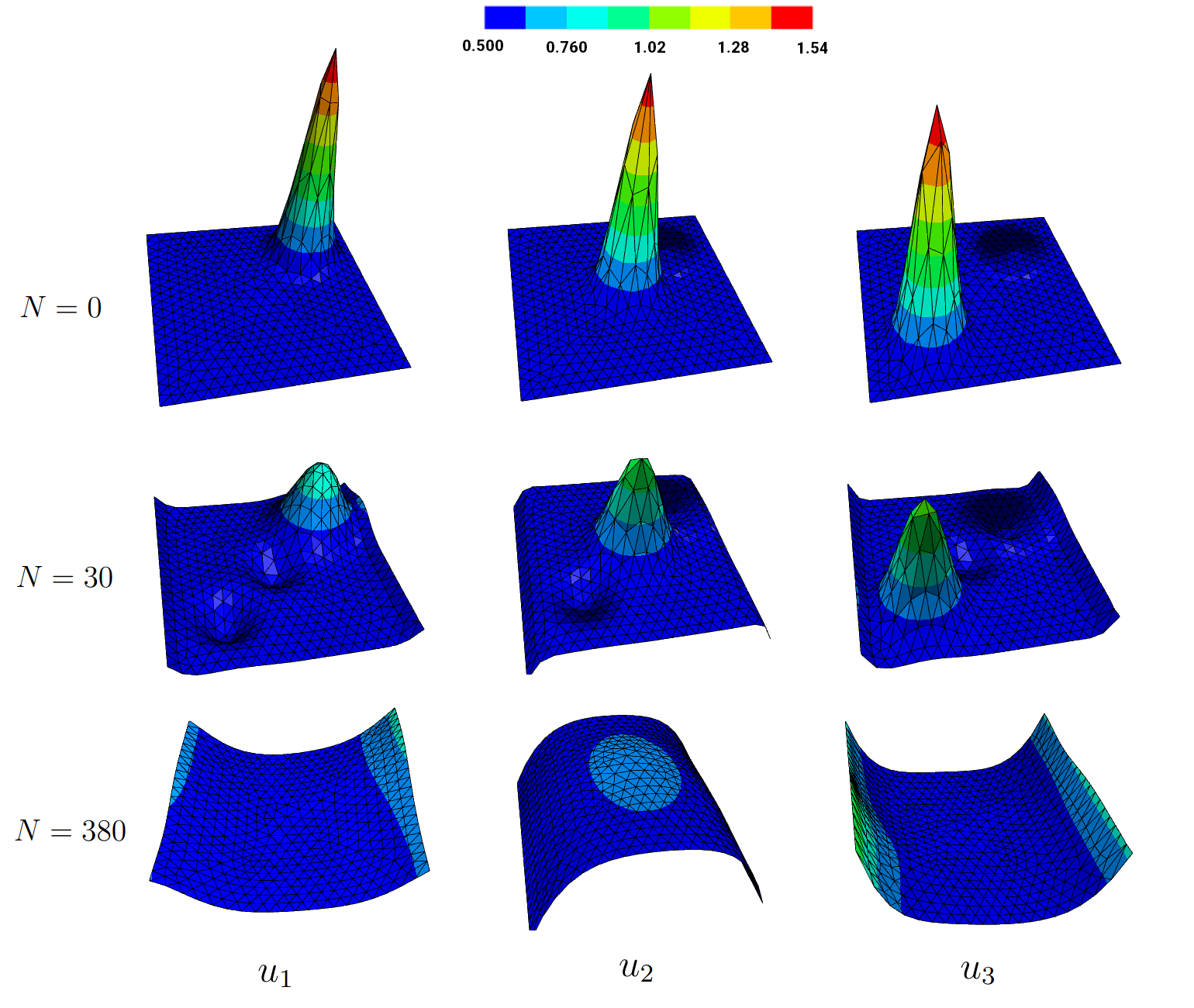}
\caption{Ion concentrations $u_1$ (left column), $u_2$ (middle column), and $u_3$ (right column) at time steps $N=0$ (top row), $N=30$ (middle row), and $N=380$ (bottom row).}
\label{fig1}
\end{figure}

\begin{figure}[ht]
\centering
\includegraphics[width=150mm]{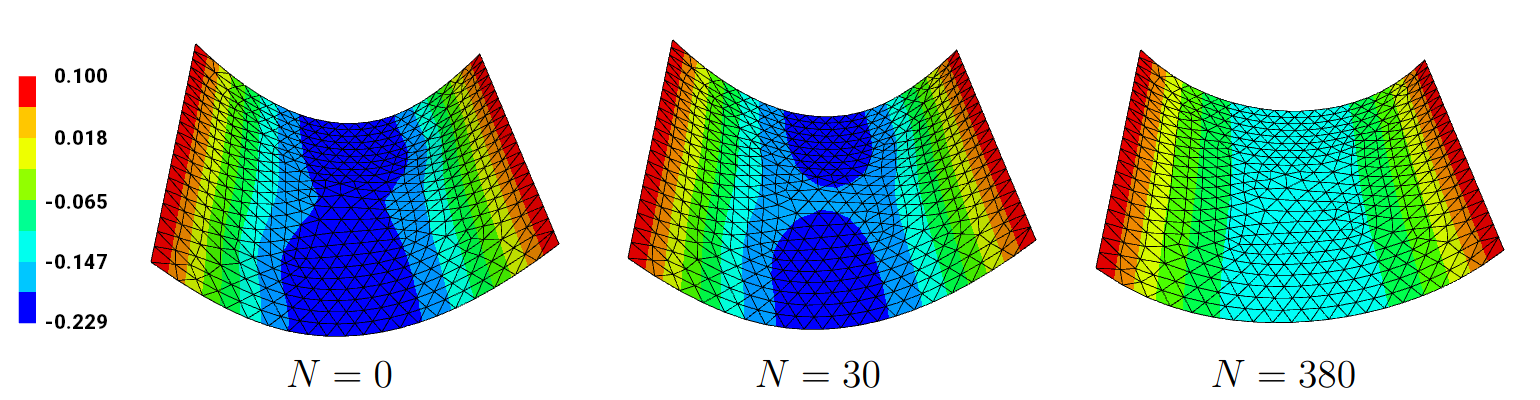}
\caption{Electric potential $\Phi$ at time steps $N=0$ (left), $N=30$ (middle), and $N=380$ (right).}
\label{fig2}
\end{figure}

%%%%%%%%%%%%%%%%%%%%%%%%%%%%%%%%%%%%%%%%%%%%%%%%%%%%%%%%%%%%%%%%%

\end{document}